\documentclass[11pt]{amsart}

 \usepackage[nobysame]{amsrefs}

 \AtBeginDocument{\def\MR#1{}}

\usepackage{currfile}

\usepackage[english]{babel} 
\usepackage[UKenglish]{datetime}

\usepackage{amscd}
\usepackage{amsfonts}
\usepackage{amsmath}
\usepackage{amssymb}
\usepackage{amsthm}
\usepackage{amsxtra}
\usepackage{cases}
\usepackage{chngcntr}
\usepackage{cite}
\usepackage{color}
\usepackage{graphicx}
\usepackage{latexsym}
\usepackage{mathtools}
\usepackage{microtype}
\usepackage{qsymbols}
\usepackage[active]{srcltx}
\usepackage{url}
\usepackage{verbatim} 

\usepackage{enumitem}
\usepackage{cleveref}

\usepackage[charter]{mathdesign}

\newcommand{\bbfont}{\mathbb}

\usepackage{ifthen}

\makeatletter
\newcommand{\tfs}[1]
{
\ifthenelse{\equal{\f@shape}{n}}{\ensuremath{\mathrm{#1}}}
	{\ifthenelse{\equal{\f@shape}{sc}}{\ensuremath{\mathrm{#1}}}
		{\ifthenelse{\equal{\f@shape}{it}}{\ensuremath{\mathit{#1}}}
			{\ifthenelse{\equal{\f@shape}{sl}}{\ensuremath{\mathit{#1}}}{}	
			}
		}
	}
}
\makeatother

\makeatletter
\newcommand{\btfs}[1]
{
\ifthenelse{\equal{\f@shape}{n}}{\ensuremath{\mathrm{#1}}}
	{\ifthenelse{\equal{\f@shape}{sc}}{\ensuremath{\mathrm{#1}}}
		{\ifthenelse{\equal{\f@shape}{it}}{\ensuremath{\mathit{#1}}}
			{\ifthenelse{\equal{\f@shape}{sl}}{\ensuremath{\mathit{#1}}}{}	
			}
		}
	}
}
\makeatother

\newcommand{\RR}{{\bbfont R}}
\newcommand{\NN}{{\bbfont N}}

\newcommand{\ulp}{{\textup{(}}}
\newcommand{\urp}{{\textup{)}}}

\newcommand{\uppars}[1]{\ulp #1\urp}

\newcommand{\abs}[1]{{\lvert #1 \rvert}}
\newcommand{\norm}[1]{{\lVert #1 \rVert}}

\newcommand{\braces}[1]{{\{ #1\}}}

\newcommand{\lrnorm}[1]{{\left\lVert #1 \right\rVert}}

\newcommand{\bignorm}[1]{{\big\lVert #1 \big\rVert}}

\newcommand{\desset}[1]{\braces{\,#1\,}}

\theoremstyle{plain}

\newtheorem{theorem}{Theorem}[section]
\newtheorem{proposition}[theorem]{Proposition}
\newtheorem{lemma}[theorem]{Lemma}

\newtheorem*{theorem*}{Theorem}
\newtheorem*{proposition*}{Proposition}
\newtheorem*{lemma*}{Lemma}
\newtheorem*{corollary*}{Corollary}
\newtheorem*{conjecture*}{Conjecture}
\newtheorem*{assumption*}{Assumption}
\newtheorem*{hypothesis*}{Hypothesis}
\newtheorem*{problem*}{Problem}
\newtheorem*{task*}{Task}
\newtheorem*{addendum*}{Addendum}
\newtheorem*{idea*}{Idea}
\newtheorem*{suggestion*}{Suggestion}
\newtheorem*{context*}{Context}
\newtheorem*{exercise*}{Exercise}

\theoremstyle{definition}

\newtheorem{definition}[theorem]{Definition}
\newtheorem{example}[theorem]{Example}
\newtheorem{examples}[theorem]{Examples}
\newtheorem{remark}[theorem]{Remark}

\newtheorem{questions}[theorem]{Questions}

\newtheorem*{definition*}{Definition}
\newtheorem*{example*}{Example}
\newtheorem*{examples*}{Examples}
\newtheorem*{remark*}{Remark}
\newtheorem*{question*}{Question}
\newtheorem*{questions*}{Questions}

\setlist[enumerate,1]{label=\textup{(\arabic*)},ref=\arabic*}
\setlist[enumerate,2]{label=\textup{(\alph*)},ref=\arabic{enumi}.\alph*}
\setlist[enumerate,3]{label=\textup{(\roman*)},ref=\arabic{enumi}.\alph{enumii}.\roman*}
\setlist[enumerate,4]{label=\textup{(\Alph*)},ref=\arabic{enumi}.\alph{enumii}.\roman{enumiii}.\Alph*}

\crefname{theorem}{Theorem}{Theorems}
\crefname{proposition}{Proposition}{Propositions}
\crefname{lemma}{Lemma}{Lemmas}
\crefname{corollary}{Corollary}{Corollaries}
\crefname{conjecture}{Conjecture}{Conjectures}
\crefname{definition}{Definition}{Definitions}
\crefname{example}{Example}{Examples}
\crefname{examples}{Examples}{Examples}
\crefname{remark}{Remark}{Remarks}
\crefname{assumption}{Assumption}{Assumptions}
\crefname{hypothesis}{Hypothesis}{Hypotheses}
\crefname{question}{Question}{Questions}
\crefname{problem}{Problem}{Problems}
\crefname{task}{Task}{Tasks}
\crefname{addendum}{Addendum}{Addenda}
\crefname{idea}{Idea}{Ideas}
\crefname{suggestion}{Suggestion}{Suggestions}
\crefname{context}{Context}{Contexts}
\crefname{exercise}{Exercise}{Exercises}
\crefname{section}{Section}{Sections}
\crefname{subsection}{Section}{Sections}
\crefname{subsubsection}{Section}{Sections}
\crefname{questions}{Question}{Questions}

\crefname{equation}{equation}{equations}
\crefname{enumi}{part}{parts}
\crefname{enumii}{part}{parts}
\crefname{enumiii}{part}{parts}
\crefname{enumiv}{part}{parts}

\newcommand{\enclosepart}[1]{(#1)}
\newcommand{\partref}[1]{\enclosepart{\ref{#1}}}

\numberwithin{equation}{section}

\allowdisplaybreaks 

\usepackage{tikz,tikz-cd}
\usetikzlibrary{matrix, arrows}

\makeatletter
\tikzset{
	edge node/.code={%
		\expandafter\def\expandafter\tikz@tonodes\expandafter{\tikz@tonodes #1}}}
\makeatother
\tikzset{
	subseteq/.style={
		draw=none,
		edge node={node [sloped, allow upside down, auto=false]{$\subseteq$}}},
	Subseteq/.style={
		draw=none,
		every to/.append style={
			edge node={node [sloped, allow upside down, auto=false]{$\subset$}}}}
}

\newcommand{\fact}[1]{{\overline{#1}}}
\newcommand{\factfact}[1]{{\fact{\fact{#1}}}}

\renewcommand{\aa}{A}
\newcommand{\aaone}{{\aa^{1}}}
\newcommand{\vla}{A}
\newcommand{\vlaone}{{\vla}^1}
\newcommand{\vlaonepos}{{\vla^{1+}}}
\newcommand{\ba}{A}

\newcommand{\alwedge}{\owedge}
\newcommand{\alvee}{\ovee}
\newcommand{\absalg}{A}
\newcommand{\absalgtwo}{B}

\newcommand{\type}{(\mathcal F,\rho)}
\newcommand{\termalg}[1]{T_{\type}(#1)}
\newcommand{\termalgS}{{\termalg{S}}}
\newcommand{\infset}{S_{\aleph_0}}
\newcommand{\termalginfset}{{\termalg{\infset}}}

\newcommand{\catfont}[1]{\textup{\fontfamily{qag}\selectfont{\footnotesize #1}}}
\newcommand{\catfontscript}[1]{\textup{\fontfamily{qag}\selectfont{\tiny{#1}}}}

\newcommand{\catsymbol}{\catfont{Cat}}
\newcommand{\catone}{\catsymbol_1}
\newcommand{\cattwo}{\catsymbol_2}
\newcommand{\catthree}{\catsymbol_3}

\newcommand{\catsymbolscript}{\catfontscript{Cat}}
\newcommand{\catonescript}{{\catsymbolscript}_1}
\newcommand{\cattwoscript}{{\catsymbolscript}_2}
\newcommand{\catthreescript}{\catsymbolscript_3}

\newcommand{\SET}{\catfont{Set}}
\newcommand{\VS}{\catfont{VS}}
\newcommand{\VL}{\catfont{VL}}
\newcommand{\VLA}{\catfont{VLA}}
\newcommand{\VLAONE}{{\VLA}^1}
\newcommand{\VLAONEPOS}{{\VLA}^{1+}}

\newcommand{\LAT}{\catfont{Lat}}

\newcommand{\SETscript}{\catfontscript{Set}}
\newcommand{\VSscript}{\catfontscript{VS}}
\newcommand{\VLscript}{\catfontscript{VL}}
\newcommand{\VLAscript}{\catfontscript{VLA}}
\newcommand{\VLAONEscript}{{\VLAscript}^1}
\newcommand{\VLAONEPOSscript}{{\VLAscript}^{1+}}

\newcommand{\LATscript}{\catfontscript{Lat}}

\newcommand{\ALG}{\catfont{Alg}}
\newcommand{\ALGONE}{{\ALG}^1}
\newcommand{\MET}{\catfont{Met}}
\newcommand{\COMMET}{\catfont{ComMet}}
\newcommand{\BA}{\catfont{BA}}
\newcommand{\ALGRHO}{\catfont{AbsAlg}_{\type}}
\newcommand{\ALGRHOSIGMA}{\catfont{AbsAlg}_{\type; \Sigma}}
\newcommand{\GRP}{\catfont{Grp}}
\newcommand{\ABGRP}{\catfont{AbGrp}}

\newcommand{\ALGscript}{\catfontscript{Alg}}
\newcommand{\ALGONEscript}{{\ALGscript}^1}
\newcommand{\METscript}{\catfontscript{Met}}
\newcommand{\COMMETscript}{\catfontscript{ComMet}}
\newcommand{\BAscript}{\catfontscript{BA}}
\newcommand{\ALGRHOscript}{\catfontscript{AbsAlg}_{\type}}
\newcommand{\ALGRHOSIGMAscript}{\catfontscript{AbsAlg}_{\type; \Sigma}}
\newcommand{\GRPscript}{\catfontscript{Grp}}
\newcommand{\ABGRPscript}{\catfontscript{AbGrp}}

\newcommand{\freeletter}{\textup{F}}
\newcommand{\Free}[3]{\freeletter_{#1}^{#2}\!\left[#3\right]}

\newcommand{\FCATONECATTWO}[1]{\Free{\catonescript}{\cattwoscript}{#1}}
\newcommand{\FCATTWOCATTHREE}[1]{\Free{\cattwoscript}{\catthreescript}{#1}}
\newcommand{\FCATONECATTHREE}[1]{\Free{\catonescript}{\catthreescript}{#1}}

\newcommand{\FSETVS}[1]{\Free{\SETscript}{\VSscript}{#1}}
\newcommand{\FSETVL}[1]{\Free{\SETscript}{\VLscript}{#1}}
\newcommand{\FSETVLA}[1]{\Free{\SETscript}{\VLAscript}{#1}}
\newcommand{\FSETVLAONE}[1]{\Free{\SETscript}{\VLAONEscript}{#1}}
\newcommand{\FSETVLAONEPOS}[1]{\Free{\SETscript}{\VLAONEPOSscript}{#1}}

\newcommand{\FVSVL}[1]{\Free{\VSscript}{\VLscript}{#1}}
\newcommand{\FVSVLA}[1]{\Free{\VSscript}{\VLAscript}{#1}}
\newcommand{\FVSVLAONE}[1]{\Free{\VSscript}{\VLAONEscript}{#1}}
\newcommand{\FVSVLAONEPOS}[1]{\Free{\VSscript}{\VLAONEPOSscript}{#1}}

\newcommand{\FVLVLA}[1]{\Free{\VLscript}{\VLAscript}{#1}}
\newcommand{\FVLVLAONE}[1]{\Free{\VLscript}{\VLAONEscript}{#1}}
\newcommand{\FVLVLAONEPOS}[1]{\Free{\VLscript}{\VLAONEPOSscript}{#1}}

\newcommand{\FVLAVLAONE}[1]{\Free{\VLAscript}{\VLAONEscript}{#1}}
\newcommand{\FVLAVLAONEPOS}[1]{\Free{\VLAscript}{\VLAONEPOSscript}{#1}}

\newcommand{\FVLAONEVLAONE}[1]{\Free{\VLAONEscript}{\VLAONEscript}{#1}}
\newcommand{\FVLAONEVLAONEPOS}[1]{\Free{\VLAONEscript}{\VLAONEPOSscript}{#1}}

\newcommand{\FLATVL}[1]{\Free{\LATscript}{\VLscript}{#1}}
\newcommand{\FLATVLA}[1]{\Free{\LATscript}{\VLAscript}{#1}}
\newcommand{\FLATVLAONE}[1]{\Free{\LATscript}{\VLAONEscript}{#1}}
\newcommand{\FLATVLAONEPOS}[1]{\Free{\LATscript}{\VLAONEPOSscript}{#1}}

\newcommand{\FMETCOMMET}[1]{\Free{\METscript}{\COMMETscript}{#1}}
\newcommand{\FSETBA}[1]{\Free{\SETscript}{\BAscript}{#1}}
\newcommand{\FALGALGONE}[1]{\Free{\ALGscript}{\ALGONEscript}{#1}}
\newcommand{\FSETALGRHO}[1]{\Free{\SETscript}{\ALGRHOscript}{#1}}
\newcommand{\FSETALGRHOSIGMA}[1]{\Free{\SETscript}{\ALGRHOSIGMAscript}{#1}}
\newcommand{\FGRPABGRP}[1]{\Free{\GRPscript}{\ABGRPscript}{#1}}

\begin{document}


\title [Free vector lattices and free vector lattice algebras]{Free vector lattices and free vector lattice algebras}

\author{Marcel de Jeu}
\address{\!\!Mathematical Institute, Leiden University, P.O.\ Box 9512, 2300 RA Leiden, the Netherlands;
and
Department of Mathematics and Applied Mathematics, University of Pretoria, Cor\-ner of Lynnwood Road and Roper Street, Hatfield 0083, Pretoria, South Africa}
\email{mdejeu@math.leidenuniv.nl}


\dedicatory{Dedicated to the memory of Coenraad Labuschagne}





\begin{abstract}
We show how the existence of various free vector lattices and free vector lattice algebras can be derived from a theorem on equational classes in universal algebra. A discussion about free $\!f\!$-algebras over non-empty sets is given, where the main issues appear to be open. It is indicated how the existence results for free vector lattices and vector lattice algebras can be used for easy proofs of existence results for free Banach lattices and free Banach lattice algebras. A detailed exposition of the necessary material from universal algebra is included.
\end{abstract}

\maketitle


\section{Introduction and overview}\label{sec:introduction_and_overview}

\noindent In recent years, there has been a growing interest in free Banach lattices. Definitions have been given of a free Banach lattice over a set (see \cite{de_pagter_wickstead:2015}), over a Banach space (see \cite{aviles_rodriguez_tradacete:2018,aviles_tradacete_villanueva:2019,troitsky:2019}), and over a lattice (see \cite{aviles_rodriguez-abellan:2019}). These objects have been shown to exist, and properties beyond their mere existence have been studied.

The starting point for the existence proofs in these papers is a concrete model for a free object that has been obtained earlier. The most basic of these concrete models appears to be the usual model for the free vector lattice over a non-empty set $S$ as a sublattice of $\mathbb R^{\mathbb R^S}$; see \cite{baker_1968} or \cite{bleier:1973}, for example. In \cite{de_pagter_wickstead:2015}, this concrete model is then used in the construction of the free Banach lattice over a non-empty set $S$ (see \cite[Definition~4.4]{de_pagter_wickstead:2015}). Likewise, it is an ingredient in the construction of the free Banach lattice over a Banach space (see \cite[beginning of the proof of Theorem~2.5]{aviles_rodriguez_tradacete:2018}). In \cite[p.~583]{aviles_rodriguez-abellan:2019}, the existence of the free Banach lattice over a non-empty set, as established in \cite{de_pagter_wickstead:2015}, is used to construct the free Banach lattice over a vector lattice. This construction is, therefore, in the end also based on the usual concrete model for the free vector lattice over a non-empty set. In \cite{troitsky:2019}, this model is used to construct the free Banach lattice over non-empty sets again (simplifying the existence proof in \cite{de_pagter_wickstead:2015}); this, in turn, is used to construct the free Banach lattice over a Banach space (simplifying the existence proof in \cite{aviles_rodriguez_tradacete:2018}).

It seems to have escaped notice so far that there is an alternative and, as we believe, simpler way to obtain the existence of such free functional analytic objects. The general strategy is to start with the mere existence\textemdash a concrete model is not needed\textemdash of a corresponding free object in an algebraic context and then, almost as an afterthought, add the norm to the picture by using a few standard constructions. Let us give a detailed example on how to construct a free Banach lattice over a Banach space along these lines.

Suppose that $X$ is a (real) Banach space. By \cref{res:overview_algebraic}, below, there exist a vector lattice $E$ and a map $j:X\to E$ with the property that, for every vector lattice $Y$ and for every linear map $\varphi:X\to Y$, there exists a unique vector lattice homomorphism $\fact{\varphi}$ such that the diagram
\begin{equation*}\label{dia:banach_1}
\begin{tikzcd}
X\arrow[r, "j"]\arrow[dr, "\varphi"]& E\arrow[d, "\fact{\varphi}"]
\\& Y
\end{tikzcd}
\end{equation*}
is commutative. Such a vector lattice $E$ is called a free vector lattice over the vector space $X$. It is easy to see that $E$ is generated, as a vector lattice, by its subset $j(X)$. For $e\in E$, set
\[
\rho(e)\!\coloneqq \sup\left\{\norm{\fact{\Psi}(e)}\!:\!Y\!\text{ is a Banach lattice and }\Psi:X\!\to Y\text{ is a contraction}\right\}.
\]
Using the correspondence between lattice seminorms on $E$ and vector lattice homomorphisms from $E$ into Banach lattices, one easily sees that one can, equivalently, set
\[
\rho(e)\!\coloneqq \sup\left\{\sigma(e)\!:\! \sigma\!\text{ is a lattice seminorm on }\!E\! \text{ and }\sigma\!\circ\! j\text{ is contractive on }\!X\right\},
\]
thereby avoiding possible set-theoretical subtleties.

For $x\in X$, it is clear that $\rho(j(x))\leq \norm{x}$. Since the subset of $E$ on which $q$ is finite is easily seen to be a vector sublattice of $E$, and since $j(X)$ generates $E$ as a vector lattice, we conclude that $\rho$ is a lattice seminorm on $E$.\footnote{It is a non-trivial fact that $\rho$ is actually a lattice norm on $E$; this follows from \cite[Theorem~3.1]{troitsky:2019}. For the present construction to go through this is, however, not needed.} The kernel of $\rho$ is an order ideal in $E$. We let $q:E\to E/\ker \rho$ denote the quotient map. On setting $\norm{q(e)}\coloneqq \rho(e)$ for $e\in E$, the vector lattice $E/\ker \rho$ becomes a normed vector lattice. It is then immediate that $q\circ j:X\to E/\ker \rho$ is contractive.

Let $Y$ be a Banach lattice and let $\varphi: X\to Y$ be a bounded linear map. Suppose that $e\in E$ is such that $\rho(e)=0$. We claim that then also $\fact{\varphi}(e)=0$. This is clear if $\varphi=0$. When $\varphi\neq 0$, so that $\varphi/\norm{\varphi}$ is a contraction, this follows from the definition of $\rho$.
Hence there exists a unique vector lattice homomorphism $\factfact{\varphi}$ such that the diagram
\begin{equation}\label{dia:banach_2}
\begin{tikzcd}
X\arrow[r, "j"]\arrow[dr,"\varphi"]& E \arrow[r,"q"]\arrow[d, "\fact{\varphi}"]&E/\ker \rho\arrow[dl,"\factfact{\varphi}"]\\
&Y&
\end{tikzcd}
\end{equation}
is commutative.

We claim that $\factfact{\varphi}$ is bounded and that $\bignorm{\factfact{\varphi}}=\lrnorm{\varphi}$. We may suppose that $\varphi\neq 0$, so that $\varphi/\norm{\varphi}$ is a contraction. For $e\in E$, we then have
\begin{align*}
\bignorm{\factfact{\varphi}(q(e))}&=\bignorm{\fact{\varphi}(e)}\\
&=\lrnorm{\varphi}\,\lrnorm{\overline{\left(\frac{\varphi}{\norm{\varphi}}\right)}(e)}\\
&\leq \lrnorm{\varphi}\,\rho(e)\\
&=\lrnorm{\varphi}\,\norm{q(e)}.
\end{align*}
Hence $\factfact{\varphi}$ is bounded and $\bignorm{\factfact{\varphi}}\leq\lrnorm{\varphi}$. On the other hand, the fact that $\varphi=\factfact{\varphi}\circ(q\circ j)$, combined with the fact that $q\circ j$ is contractive, shows that $\lrnorm{\varphi}\leq \bignorm{\factfact{\varphi}}$. Hence $\bignorm{\factfact{\varphi}}=\lrnorm{\varphi}$, as claimed.

There exists an isometric linear map from $X$ into a Banach lattice. Indeed, the canonical embedding of $X$ into the bounded real-valued functions on the unit ball of its dual is such a map.
Take such an isometric linear embedding for $\varphi$. Then $\bignorm{\factfact{\varphi}}=\norm{\varphi}\leq 1$. Since we already know that $q\circ j$ is contractive, we have, for $x\in X$,
\begin{align*}
\norm{x}&=\norm{\varphi(x)}\\
&=\bignorm{\,[\factfact{\varphi}\circ(q\circ j)](x)\,}\\
&\leq \bignorm{\factfact{\varphi}}\norm{(q\circ j)(x)}\\
&= \norm{\varphi}\norm{(q\circ j)(x)}\\
&\leq \norm{(q\circ j)(x)}\\
&\leq \norm{x}.
\end{align*}

We conclude that $q\circ j$ is isometric.

We note that, since $(q\circ j)(X)$ generates $E/\ker \rho$ as a vector lattice, $\factfact{\varphi}: E/\ker\rho\to Y$ is uniquely determined as a vector lattice homomorphism by the requirement that $\varphi=\factfact{\varphi}\circ(q\circ j)$.

We have thus found a normed vector lattice $E/\ker\rho$ and an isometric linear map $q\circ j:X\to E/\ker\rho $ with the property that, for every Banach space $Y$, and for every bounded linear map $\varphi:X\to Y$, there exists a unique vector lattice homomorphism $\fact{\varphi}: E/\ker\rho\to Y$ such that the diagram
\begin{equation*}
\begin{tikzcd}
X\arrow[r, "q\circ j"]\arrow[dr, "\varphi"]& E/\ker\rho\arrow[d, "\fact{\varphi}"]
\\& Y
\end{tikzcd}
\end{equation*}
is commutative. Moreover, $\fact{\varphi}$ is a bounded linear operator, and $\bignorm{\factfact{\varphi}}=\norm{\varphi}$.

Let $F$ be the norm completion of $E/\ker\rho$, and set $j_F\coloneqq q\circ j$, seen as a map from $X$ into $F$. We see, removing the construction from the notation, that there exist a Banach lattice $F$ and an isometric linear map $j:X\to F$ with the property that, for every Banach lattice $Y$, and for every bounded linear map $\varphi:X\to Y$, there exists a unique vector lattice homomorphism $\varphi: F\to Y$ with $\norm{\fact{\varphi}}=\norm{\varphi}$ such that the diagram
\begin{equation*}
\begin{tikzcd}
X\arrow[r, "j"]\arrow[dr, "\varphi"]& F \arrow[d, "\fact{\varphi}"]
\\& Y
\end{tikzcd}
\end{equation*}
is commutative. Using only the existence of a free vector lattice over a vector space as a starting point, we have thus retrieved the existence of a free Banach lattice over a Banach space as was first established in \cite{aviles_rodriguez_tradacete:2018}.

The type of arguments in the preceding construction have been used earlier in the papers cited above, but this was always done using the setting in which a concrete model for the free object was constructed. The existence of the free object and the proof that a particular object was a concrete model for it came at the same time. In the above line of reasoning, however, nothing specific is used. Once one has the existence of a free vector lattice over a vector space, the rest is an argument that takes place in the category of vector lattices and in that of Banach spaces.

It is clear that this abstract approach can be used in other situations. One starts with a free object in the algebraic context, introduces an appropriate seminorm on it, divides out its kernel, and completes. For example, the free Banach lattice over a lattice from \cite{aviles_rodriguez-abellan:2019} can also be constructed along these lines, with the existence of the free vector lattice over a lattice (see \cref{res:free_objects_over_lattices}, below) as a starting point. In fact, all free objects in the algebraic context in \cref{res:overview_algebraic,res:free_objects_over_lattices}, below, can be used to construct their functional analytic counterparts. We intend to report separately on this in the future. In particular, it will then be seen that free Banach lattice algebras over non-empty sets exist, as well as unitisations of Banach lattice algebras; this solves Problems~13 and~15 in Wickstead's list \cite{wickstead:2017c}. In the case of Banach lattice algebras, it is then necessary to incorporate certain bounds into the construction. The final example in \cref{ex:free_objects}, below, may serve to make this plausible.

We hasten to add that this general abstract approach will only lead to an easy existence proof for certain free functional analytic objects. It will not inform us how to find a concrete model that enables a further study of its structure. For example, as an algebraic analogue, it is not difficult to see that free vector lattices over non-empty sets exist (this follows from a general result in universal algebra), but it requires creativity  as in \cite{bleier:1973} to show that they can be realised as lattices of functions. It is only then that it becomes clear that they are Archimedean. Our approach can, therefore, not replace the papers cited above where the structure of various free functional analytic objects is studied. The virtue of the general abstract approach is that it provides a smooth standard route to the basic existence result, after which the actual work can begin.

\smallskip

\noindent The present paper is intended to provide an algebraic basis for future existence results for free Banach lattices and Banach lattice algebras. As will become clear in the sequel, it is actually quite easy to prove that, for example, free vector lattice algebras with positive identity elements over non-empty sets exist. One merely needs to show that these are a so-called equational class, which is not too difficult, and then invoke a general theorem from universal algebra that holds for such classes to conclude the proof. It is even easier to show  the existence of free vector lattices over non-empty sets along these lines.

It seems, however, as if this possibility of invoking a ready-to-use result from universal algebra in the context of vector lattices and vector lattice algebras was once known, but has later faded into the background. In his 1973 paper \cite{bleier:1973} already mentioned above, Bleier showed that the free vector lattice over a set has the usual concrete model as a vector lattice of functions. It is instrumental for him to know a priori that such a free vector lattice exists, but he does not even find it necessary to give a reference for this existence result. There is a simple remark \emph{`If $S$ is a non-empty set, then, since the class vector lattices is equationally definable, there exists a \uppars{unique up to isomorphism} vector lattice $F$ which is free on $S$'} (see \cite[p.74]{bleier:1973}). A reference is then given to the general theory of free abstract algebras in \cite[p.143-144]{birkhoff_LATTICE_THEORY_THIRD_EDITION:1967}.
Apparently, this was sufficient in that period of time.

On the other hand, the Problems~13 and~15 on Banach lattice algebras in \cite{wickstead:2017c}, already mentioned above, were posed at a 2014 workshop at the Lorentz Center in Leiden where various senior researchers in the field of positivity were present. None of them was aware of the fact that, at the algebraic level, the existence of free vector lattice algebras over non-empty sets and of the unitisations of vector lattice algebras can easily be derived from one single theorem in universal algebra.

Therefore, apart from providing an algebraic basis for future existence results for free Banach lattices and Banach lattice algebras, this paper is also intended to re-vitalise this knowledge of universal algebra in the specific context of lattices, and in such a way that it can easily be used in other situations. The main theorem we need can be found in textbooks on universal algebra, but there it is among much more material that is not relevant for our purposes, and it may take some effort to isolate what one actually needs. Furthermore, vector lattice algebras, for example, are hard to find in such books\textemdash if at all\textemdash and we really need to recognise them as abstract algebras where certain identities are satisfied. It is not directly obvious how one can capture a partial ordering and the existence of suprema and infima in identities, but the fact that this is nevertheless possible (see \cref{res:two_types_of_lattices_are_equivalent}, below) is crucial. We have, therefore, included the details for everything we need; the paper is self-contained. Our coverage of the material on universal algebra, culminating in the existence of free objects of equational classes over non-empty sets (see \cref{res:free_algebra_with_relations_exists}, below), is an exposition of parts of a known theory. It is, however, a very selective one, aiming for the one result we need and nothing more, and tailored to the context of lattices. We also need some basic known facts about free objects in a categorical language. The remainder of the paper, as well as the blend of category theory, universal algebra and vector lattices and vector lattice algebras, appears to be new.

\smallskip

\noindent This paper is organised as follows.

In \cref{sec:free_objects}, we introduce the notion of a free object. Some examples are given, and preparations for later sections are made.

In \cref{sec:universal_algebra_part_I}, we start our exposition on universal algebra. The goal is the existence theorem for a free (abstract) algebra of a given type over a non-empty set; see \cref{res:term_algebra_is_free}, below.

\cref{sec:various_lattices_as_abstract_algebras_satisfying_identities} is concerned with a crucial point: capturing the partial ordering and the existence of infima and suprema in a lattice in identities. \cref{res:VLAONE_is_equational_class}, below, shows that the unital vector lattice algebras are precisely the abstract algebras of a certain type where a list of identities are satisfied.

After \cref{sec:various_lattices_as_abstract_algebras_satisfying_identities}, there is a need to formalise the notion of an abstract algebra `satisfying identities'. This is done in \cref{sec:universal_algebra_part_II}, where we continue our exposition on universal algebra and equational classes are introduced. \cref{res:free_algebra_with_relations_exists}, below, is the result we are after. It guarantees the existence of free abstract algebras in equational classes over non-empty sets.

In \cref{sec:free_vector_lattices_and_free_vector_lattice_algebras}, the harvest is brought in. We start by inferring the existence of free unital vector lattice algebras over non-empty sets. From this one existence result, the existence of a host of other free objects is established. Together with their interrelations, they are collected in \cref{res:overview_algebraic}, below.

In \cref{sec:free_objects_over_lattices}, the existence of a number of free objects over (not necessarily distributive) lattices is shown; see \cref{res:free_objects_over_lattices}, below. Although an independent approach is also possible, we have chosen to derive the results in this section from those in \cref{sec:free_vector_lattices_and_free_vector_lattice_algebras}.

Finally, in \cref{sec:free_f-algebras}, we consider various types of free $\!f\!$-algebras over non-empty sets. We can neither prove nor disprove that they exist.
It is motivated how the observation that the free vector lattice over a non-empty set is, in fact, Archimedean can lead one to wonder whether free vector lattice algebras are, in fact, $f\!$-algebras.

\smallskip

\noindent We conclude this section by mentioning some terminology and conventions.

All vector spaces are over the real numbers. An \emph{algebra} is an associative algebra. An algebra need not be unital. An algebra homomorphism between two unital algebras need not be unital. When convenient, a unital algebra will be denoted by $\aaone$. A \emph{vector lattice algebra}, also called a Riesz algebra in the literature, is  a vector lattice that is also an algebra such that the product of two positive elements is positive. The identity element of a unital vector lattice algebra need not be positive. A \emph{vector lattice algebra homomorphism} between vector lattice algebras is a lattice homomorphism that is also an algebra homomorphism. When convenient, a unital vector lattice algebra with a positive identity element will be denoted by $\vlaonepos$. A \emph{bi-ideal} in a vector lattice algebra is a linear subspace that is an order ideal as well as a two-sided algebra ideal.

We let $\NN_0=\{0,1,2\ldots\}$.

\section{Free objects}\label{sec:free_objects}

\noindent In this section, we review basic facts about free objects. This notion, to be defined below, can be introduced whenever a category is a subcategory of another category. In our case, the main six categories of interest are:
\begin{itemize}
	\item $\SET$: the sets with the maps as morphisms;
	\item $\VS$: the vector spaces, with the linear maps as morphisms;
	\item $\VL$: the vector lattices, with the vector lattice homomorphisms as morphisms;
	\item $\VLA$: the vector lattice algebras, with the vector lattice algebra homomorphisms as morphisms;
	\item $\VLAONE$: the unital vector lattice algebras, with the unital vector lattice algebra homomorphisms as morphisms;
	\item $\VLAONEPOS$: the unital vector lattice algebras that have a positive identity element, with the unital vector lattice algebra homomorphisms as morphisms.
\end{itemize}

There will also be an appearance of $\LAT$, the category of not necessarily distributive lattices, with lattice homomorphisms as morphism. There are two ways to define lattices. The fact that these two are equivalent (see \cref{res:two_types_of_lattices_are_equivalent}) is essential for this paper.

The main six categories of our interest can be ordered in the following chain. With the exception of $\SET$, each is a subcategory of the one to the left of it:
\begin{equation}\label{eq:chain_of_categories}
\SET \supset \VS \supset \VL \supset \VLA \supset \VLAONE \supset \VLAONEPOS.
\end{equation}
Except for $\VLAONE \supset \VLAONEPOS$, all these subcategories are non-full subcategories.

All in all, there are 15 instances of a category and a subcategory of it associated to this chain. For each of these, there is a notion of free objects. We shall now define this.

\begin{definition}\label{def:free_object}
	Suppose that $\catone$ and $\cattwo$ are categories, and that $\cattwo$ is a subcategory of $\catone$. Take an object $O_1$ of $\catone$. Then a \emph{free object over $O_1$ of $\cattwo$} is a pair $(j,\FCATONECATTWO{O_1})$, where $\FCATONECATTWO{O_1}$ is an object of $\cattwo$ and $j:O_1\to\FCATONECATTWO{O_1}$ is a morphism in $\catone$,
	with the property that, for every object $O_2$ of $\cattwo$ and every morphism $\varphi:O_1\to O_2$ of $\catone$, there exists a unique morphism  $\fact{\varphi}: \FCATONECATTWO{O_1}\to O_2$ of $\cattwo$ such that the diagram
	\begin{equation*}\label{dia:free_object}
	\begin{tikzcd}
	O_1\arrow[r, "j"]\arrow[dr, "\varphi"]& \FCATONECATTWO{O_1}\arrow[d, "\fact{\varphi}"]
	\\ & O_2
	\end{tikzcd}
	\end{equation*}
	in $\catone$ is commutative.
\end{definition}

Suppose that $(j^\prime,\FCATONECATTWO{O_1}^\prime)$ is another pair with this property. The unique morphism $\fact{j^\prime}:\FCATONECATTWO{O_1}\to \FCATONECATTWO{O_1}^\prime$ such that $j^\prime=\fact{j^\prime}\circ j$ and the unique morphism $\fact{j}:\FCATONECATTWO{O_1}^\prime\to\FCATONECATTWO{O_1}$ such that $j=\fact{j}\circ j^\prime$ are then such that $\fact{j}\circ\fact{j^\prime}$ is the identity morphism of $\FCATONECATTWO{O_1}$ and  $\fact{j^\prime}\circ\fact{j}$ is the identity morphism of $\FCATONECATTWO{O_1}^\prime$. Hence $\FCATONECATTWO{O_1}$, if it exists, is uniquely determined up to a unique compatible isomorphism. When convenient, we shall, therefore, simply say that \emph{$\FCATONECATTWO{O_1}$ exists} when there exists a pair $(j,\FCATONECATTWO{O_1})$ as above, and let $\FCATONECATTWO{O_1}$ stand for any realisation of it, the accompanying map $j$ being understood.

\begin{examples}\label{ex:free_objects}\quad
	\begin{enumerate}
		\item Let $\GRP$ denote the category of groups with the group homomorphisms as morphisms, and let $\ABGRP$ denote its full subcategory of abelian groups. Take a group $G$. Then $\FGRPABGRP{G}$ exists and is the quotient of $G$ modulo its commutator subgroup.
		\item Take a non-empty set $S$. Then $\FSETVS{S}$ exists and is the vector space of real-valued functions on $S$ with finite supports.
		\item Let $\ALG$ denote the category of algebras with the algebra homomorphisms as morphisms, and let $\ALGONE$ denote its subcategory of unital algebras with the unital algebra homomorphisms as morphisms. Take an algebra  $\aa$. Set $\aaone\coloneqq \RR\oplus \aa$, as a vector space direct sum, and supply it with the usual structure of a unital algebra by setting $(\lambda,a)\cdot(\mu,b)\coloneqq(\lambda\mu,\lambda b + \mu a +ab)$ for $\lambda,\mu\in\RR$ and $a,b\in\aa$. Then $\FALGALGONE{\aa}$ exists and equals $\aaone$.
		\item Take a vector lattice algebra $\vla$.  Set $\vlaone\coloneqq \RR\oplus \vla$ as a vector lattice direct sum, and supply it with the structure of a unital associative algebra as above. Then $\vlaone$ is a unital vector lattice algebra. It is, however, not generally the free unital vector lattice algebra over $\vla$. The problem is that the natural factoring map $\fact{\varphi}$, although a unital algebra homomorphism, need not be a lattice homomorphism. It is nevertheless true that the unitisation $\FVLAVLAONE{\vla}$ of $\vla$ exists; see \cref{res:overview_algebraic}.		
		\item Take a non-empty set $S$. Contrary to the previous example, in this case there does not even appear to be a natural (flawed) Ansatz for the free unital vector lattice algebra $\FSETVLAONE{S}$ over $S$. The fact that it nevertheless exists (see \cref{res:overview_algebraic}) is the foundation on which the other existence results in this paper are built.
		\item Let $\MET$ be the category of metric spaces with continuous maps, and let $\COMMET$ be its full subcategory of complete metric spaces. Take a metric space $M$. Then $\FMETCOMMET{M}$ exists and is the metric completion of $M$.
		\item Let $\BA$ be the category of Banach algebras with the continuous algebra homomorphisms as morphisms. Consider a set $\{s\}$ with one element. Then $\FSETBA{\{s\}}$ does not exist. To prove this, suppose, to the contrary, that there exist a Banach algebra $\FSETBA{\{s\}}$ and a map $j:\{s\}\to\FSETBA{\{s\}}$ such that, for every Banach algebra $\ba$ and every map $\varphi: \{s\}\to \ba$, there exists a unique continuous algebra homomorphism $\fact{\varphi}$ such that the diagram
		\begin{equation*}
		\begin{tikzcd}
		\{s\}\arrow[r, "j"]\arrow[dr, "\varphi"]& \FSETBA{\{s\}}\arrow[d, "\fact{\varphi}"]
		\\ & \ba
		\end{tikzcd}
		\end{equation*}
		is commutative. For $\ba$, we take the Banach algebra of the real numbers and, for every $x>0$ in $\RR$, we define the map $\varphi_x:\{s\}\to\RR$ by setting $\varphi(s)\coloneqq x$. Take $x>0$. Then, for every $n\geq 0$, we have
		\begin{align*}
		x^n&=\norm{[\varphi_x(s)]^n}\\
		&= \norm{[\fact{\varphi}_x (j(s))]^n}\\
		&= \norm{\fact{\varphi}_x ([j(s)]^n)}\\
		&\leq \norm{\fact{\varphi}_x}\,\norm{[j(s)]^n}\\
		&\leq \norm{\fact{\varphi}_x}\,\norm{j(s)}^n.
		\end{align*}
		On letting $n$ tend to infinity, we see that we must have $\norm{j(s)}\geq x$. Since $x>0$ is arbitrary, this is impossible.
	\end{enumerate}
\end{examples}

We shall suppose for the remainder of this paper that the objects of categories are sets.

\begin{remark}\label{rem:embedding}
	Suppose, in the setting of \cref{def:free_object}, that $\FCATONECATTWO{O_1}$ exists and that, for each pair of different elements $x,y\in O_1$, there exists an object $O_2$ of $\cattwo$ and a morphism $\varphi: O_1\to O_2$ of $\catone$ such that $\varphi(x)\neq \varphi(y)$. Then the map $j:O_1\to\FCATONECATTWO{O_1}$ is clearly injective. The converse is obviously true because $\FCATONECATTWO{O_1}$ is an object of $\cattwo$ and the injective morphism $j$ of $\catone$ then separates the elements of $O_1$ all at once.  We shall often use this observation and collect a number of elementary facts in this vein, where actually one injective map $\varphi:O_1\to O_2$ already separates all elements of $O_1$.
\end{remark}	

\begin{lemma}\label{res:embeddings}\quad
	\begin{enumerate}
		\item Let $S$ be a non-empty set. There exists a vector space $V$ and an injective map $\varphi:S\to V$.\label{res:embedding_1}
		\item Let $V$ be a vector space. There exist a vector lattice $E$ and an injective linear map $\varphi:V\to E$.\label{res:embedding_2}
		\item Let $E$ be a vector lattice. There exist a commutative vector lattice algebra $\vlaonepos$ with a positive identity element and an injective vector lattice homomorphism $\varphi: E\to \vlaonepos$.\label{res:embedding_3}
		\item Let $\vla$ be a vector lattice algebra. There exist a unital vector lattice algebra $\vlaonepos$ with a positive identity element and an injective vector lattice algebra homomorphism $\varphi: \vla\to \vlaonepos$.\label{res:embedding_4}
	\end{enumerate}
\end{lemma}

The parts of \cref{res:embeddings} can be combined to see that, for example, there is always an injective linear map from a given vector space into a commutative vector lattice algebra with a positive identity element. A number of (combined) inclusions of categories $\catone\supset\cattwo$ from the chain \cref{eq:chain_of_categories} can thus be `reversed' in the sense that every object of $\catone$ embeds, via a morphism in $\catone$, into an object of $\cattwo$. Since it is not true that the identity of every unital vector lattice algebra is positive, there is no general `reversal' for the inclusion $\VLAONE\supset\VLAONEPOS$.

\begin{proof}[Proof of \cref{res:embeddings}]
	For part~\partref{res:embedding_1}, we take for $V$ the real-valued functions on $S$, together with the canonical map from $S$ into $V$.\\
	For part~\partref{res:embedding_2}, we let $V^{\#}$ denote set of all linear functionals on $V$. For $E$ we take the vector lattice of all real-valued functions on $V^{\#}$. The canonical map from $V$ into $E$ is linear and injective.
	
	For part~\partref{res:embedding_3}, we first make $E$ into a commutative vector lattice algebra by supplying it with the zero multiplication. Subsequently, we apply the usual unitisation procedure to that vector lattice algebra. All in all, we take the vector lattice direct sum $\vlaonepos\coloneqq\RR\oplus E$, supplied with the multiplication $(\lambda,x)\cdot(\mu,y)\coloneqq (\lambda\mu,\lambda y + \mu x)$ for $\lambda,\mu\in\RR$ and $x,y\in E$, together with the canonical map from $E$ into $\vlaonepos$.
	
	For part~\partref{res:embedding_4}, we take the vector lattice direct sum $\vlaonepos\coloneqq\RR\oplus A$, supplied with the usual multiplication to make it into a unital vector lattice algebra with a positive identity element, together with the canonical map from $\vla$ into $\vlaonepos$.
\end{proof}

\begin{remark}\label{rem:generating}
	In our situations of interest, the objects of the category $\cattwo$ in \cref{def:free_object} are sets with operations. This implies that $\FCATONECATTWO{O_1}$, if it exists, must be generated, in the sense of $\cattwo$, by $j(O_1)$. Take a non-empty set $S$, for example. If $\FSETVL{S}$ exists, then it must be generated, as a vector lattice, by its subset $j(S)$. The reason is simply that the vector sublattice that is generated by $j(S)$ and the restricted factoring map $\fact{\varphi}$ obviously also have the required universal property. The essential uniqueness of such a pair then implies that this vector sublattice must coincide with $\FSETVL{S}$. As another example, take a vector space $V$.  If $\FVSVLAONEPOS{V}$ exist, then it must be generated, as a unital vector lattice algebra, by (its identity element and) its subspace $j(V)$. We shall often use this observation.
\end{remark}

\begin{remark}\label{rem:composition}
	Let $\catone\supset\cattwo\supset\catthree$ be a chain of categories. Take an object $O_1$ of $\catone$, and suppose that $\FCATONECATTWO{O_1}$ exists in $\cattwo$, with accompanying map $j_{12}:O_1\to\FCATONECATTWO{O_1}$. Suppose that $\FCATTWOCATTHREE{\FCATONECATTWO{O_1}}$ exists in $\catthree$, with accompanying map $j_{23}:\FCATONECATTWO{O_1}\to\FCATTWOCATTHREE{\FCATONECATTWO{O_1}}$. It is easy to see that  $\FCATONECATTHREE{O_1}$ then also exists. In fact, one can take $\FCATONECATTHREE{O_1}\coloneqq\FCATTWOCATTHREE{\FCATONECATTWO{O_1}}$ and $j_{13}\coloneqq j_{23}\circ j_{12}$ as accompanying map $j_{13}:O_1\to \FCATONECATTHREE{O_1}$
\end{remark}

\begin{remark}
	Let $\catone$  be a category, and let $\cattwo$ be a subcategory. Suppose that $\FCATONECATTWO{O_1}$ exists for every object $O_1$ of $\catone$. Since $\FCATONECATTWO{O_1}$ is not uniquely determined, there is no natural functor that assigns `the' free object $\FCATONECATTWO{O_1}$ of $\cattwo$ to $O_1$. This can be remedied to some extent, as follows. Suppose that, for each object $O_1$ of $\catone$, a free object $\FCATONECATTWO{O_1}$ of $\cattwo$ over $O_1$ has been chosen, together with its accompanying map $j: O_1\to \FCATONECATTWO{O_1}$. Suppose that $O_1^\prime$ is an object of $\catone$, and that $\varphi:O_1\to O_1^\prime$ is a morphism in $\catone$. For the chosen free object $\FCATONECATTWO{O_1^\prime}$ and accompanying map $j^\prime: O_1^\prime\to \FCATONECATTWO{O_1^\prime}$, there exists a unique morphism $\fact{\varphi}:\FCATONECATTWO{O_1}\to\FCATONECATTWO{O_1^\prime}$ of $\cattwo$ such that $(j^\prime\circ\varphi)=\fact{\varphi}\circ j$. Then an actual functor from $\catone$ to $\cattwo$ is defined by sending an object $O_1$ to the chosen free object $\FCATONECATTWO{O_1}$ of $\cattwo$, and a morphism $\varphi:O_1\to O_1^\prime$ of $\catone$ to its associated morphism $\fact{\varphi}: \FCATONECATTWO{O_1}\to\FCATONECATTWO{O_1^\prime}$ of $\cattwo$.
\end{remark}

\begin{remark}\label{rem:terminology}
	There appears to be no general agreement about the terminology for the objects $\FCATONECATTWO{O_1}$ from \cref{def:free_object}. A different way to look at the pairs $(j,\FCATONECATTWO{O_1}$, leading to a different terminology, is as follows. Take an object $O_1$ of $\catone$, and consider the pairs $(\varphi, O_2)$, where $O_2$ is an object of $\cattwo$ and $\varphi:O_1\to O_2$ is a morphism in $\catone$. We form a new category that consists of all such pairs, and where a morphism from a pair $(\varphi,O_2)$ to a pair $(\varphi^\prime, O_2^\prime)$ is a morphism $\psi$ of $\cattwo$ such that the diagram
	\begin{equation*}
	\begin{tikzcd}
	O_1\arrow[r, "\varphi"]\arrow[dr, "\varphi^\prime"]& O_2\arrow[d, "\psi"]
	\\& O_2^\prime
	\end{tikzcd}
	\end{equation*}
	is commutative. The pairs $(j,\FCATONECATTWO{O_1})$ from \cref{def:free_object} are then precisely the universally repelling objects (also known as the initial objects) of this new category. From this viewpoint, it is natural to speak of a \emph{universal} object over $O_1$ of $\cattwo$. This term is used in several places in the literature; see \cite[p.83]{mac_lane:CATEGORIES_FOR_THE_WORKING_MATHEMAICIAN:1998} or \cite[p.153]{aliprantis_langford_1984}, for example. In the terminology of  \cite[p.179]{herrlich_strecker_CATEGORY_THEORY_THIRD_EDITION:2007}, $(j,\FCATONECATTWO{O_1})$ is called a \emph{$\cattwo$-reflection of $O_1$}. In the terminology of \cite[Definition~2.10]{helemskii:2013}, $\FCATONECATTWO{O_1}$ is a \emph{free object with base $O_1$}. The terminology in \cite[Definition~8.22]{adamek_herrlich_strecker_ABSTRACT_AND_CONCRETE_CATEGORIES_THE_JOY_OF_CATS:2006} agrees with ours. The same is true for the overview paper by Pestov \cite{pestov:1993} and many titles in its biography; see \cite{pestov:1993}. It is clear from this source \cite{pestov:1993} that, in a topological or analytical context, `free' is the prevailing term. Since an analytical context is, in the end, our main motivation for the present work, and since it is also used in the papers \cite{aviles_rodriguez_tradacete:2018, aviles_rodriguez-abellan:2019,aviles_tradacete_villanueva:2019,bleier:1973,de_pagter_wickstead:2015} that are directly related to the present paper, we have chosen to adapt this too.	
	
\end{remark}

\section{Universal algebra: part I}\label{sec:universal_algebra_part_I}

\noindent In this section, we review the first part of the material from universal algebra that we need. It is largely based on the exposition in \cite{bergman_UNIVERSAL_ALGEBRA:2012}. Our treatment is slightly different in the sense that we prefer to speak of constants instead of 0-ary operations, and that we have singled them out in definitions. There is then no longer any need for conventions to be in force when a definition `degenerates' for an `operation' that does not have variables at all.  We also speak of an `abstract algebra' rather than of an `algebra', since we want to keep our convention in force that the latter term refers to an associative algebra over the real numbers. Since both notions do actually occur in one context, it seems unavoidable to make such a distinction.

\begin{definition}
	Suppose that $\mathcal F$ is a non-empty (possibly infinite) set, and that $\rho:\mathcal F\to\NN_0$ is a map. Then the pair $\type$ is called a \emph{type}. Let $\absalg$ be a non-empty set and suppose that, for each $f\in \mathcal F$, the following is given:
	\begin{enumerate}
		\item when $\rho(f)=0$: an element $f^\absalg$ of $\absalg$;
		\item when $\rho(f)\geq 1$: a map $f^\absalg:\absalg^{\rho(f)}\to\absalg$.
	\end{enumerate}
	We set $\mathcal F^\absalg\coloneqq\{\,f^\absalg : f\in \mathcal F\,\}$. The pair $\langle\absalg,\mathcal F^\absalg\rangle$ is then called an \emph{abstract algebra of type $(\mathcal F,\rho$)}. The elements of $\mathcal F$ are called \emph{operation symbols}. The elements $f^\absalg$ of $\absalg$ for those $f\in\mathcal F$ such that $\rho(f)=0$ are called the \emph{constants of $\absalg$}, and the $\rho(f)$-ary maps $f^\absalg:\absalg^{\rho(f)}\to\absalg$ for those $f$ such that $\rho(f)\geq 1$ are called the \emph{operations on $\absalg$}.
\end{definition}
When everything else is clear from the context, we shall also simply refer to $\absalg$ as an abstract algebra, the rest being tacitly understood.

Suppose $\absalgtwo$ is a non-empty subset $\absalg$ that contains the constants of $\absalg$ and such that
$f^A(B^{\rho(f)})\subseteq B$ for all $f\in\mathcal F$ such that $\rho(f)\geq 1$. Supplied with the constants of $\absalg$ and the restricted operations on $\absalg$, $B$ is then called an \emph{abstract subalgebra} of $\absalg$. It is of the same type $\type$ as $\absalg$.

Let $\type$ be a type. Suppose that $I$ is a non-empty index set and that, for each $i\in I$, $\absalg_i$ is an abstract algebra of type $\type$. Then the product $\prod_{i\in I} A_i$ becomes an abstract algebra of type $\type$ in the obvious coordinate-wise way; it is then called the \emph{abstract product algebra of the $\absalg_i$}.

\begin{example}\label{ex:first_group_example}
	Take $\mathcal F =\{f_0,f_1,f_2\}$ for some symbols $f_0$, $f_1$, and $f_2$. Set $\rho(f_0)\coloneqq 0$, $\rho(f_1)\coloneqq 1$, and $\rho(f_2)\coloneqq 2$.
	Let $G$ be a group.
	
	\begin{enumerate}
		\item Set $f_0^G\coloneqq e$, where $e$ is the identity element of $G$; set $f_1^G(x)=x^{-1}$ for $x\in G$; and set $f_2^G(x,y)\coloneqq xy$ for $x,y\in G$. Then $\langle G, \{f_0^G,f_1^G,f_2^G \}\rangle$ is an abstract algebra of type $\type$.
		\item Take an element $x_0$ of $G$. Set $\tilde f_0^G\coloneqq x_0$; set $\tilde f_1^G(x)=x^7$ for $x\in G$; and set $\tilde f_2^G(x,y)\coloneqq x^2 y x^{-1} y^3 x^2$ for $x,y\in G$. Then $\langle G,\{\tilde f_0^G,\tilde f_1^G, \tilde f_2^G\}\rangle$ is an abstract algebra of type $\type$.
	\end{enumerate}
\end{example}

The notation as used in part (1) of \cref{ex:first_group_example} is not very suggestive. Given a group $G$, it would be more natural to simply speak of the associated abstract algebra $\langle G,\{e,{\,}^{-1}\,,\,\cdot\,\}\rangle$, where the type $\type$ with an underlying set $\mathcal F$ of cardinality 3, and the map $\rho: \mathcal F\to\{e,{\,}^{-1}\,,\,\cdot\,\}$ (the set containing the constant and the two actual operations on $G$) understood to be evident from the context. Given two groups $G_1$ and $G_2$, it would then, strictly speaking, be necessary to write $\langle G_1,\{e^{G_1},{{\,}^{-1}}^{G_1}\,,\,\cdot^{G_1}\,\}\rangle$ and $\langle G_2,\{e^{G_2},{{\,}^{-1}}^{G_2}\,,\,\cdot^{G_2}\,\}\rangle$. When working with concrete examples we shall omit these superscripts. For example, let $V$ be a vector space. Then there is a naturally associated abstract algebra $\langle V,\,\{0,\,+,\,\mathrm{ADDINV},\,\{\,m_\lambda:\lambda\in\RR\,\}    \}\rangle$. The unspecified set $\mathcal F$ is now uncountable, and to its elements correspond a constant $0$ of $V$, an obvious binary operation $+$, a unary operation $\mathrm{ADDINV}$ that sends $x\in V$ to $-x$, and, for every $\lambda\in \RR$, a unary operation $m_\lambda$ that sends $x\in V$ to $\lambda x$. It is then also clear what the function $\rho:\mathcal F\to\NN_0$ is; it takes the values 0,1, and 2. When $W$ is another vector space, we denote its associated abstract algebra by $\langle W,\,\{0,\,+,\,\mathrm{ADDINV},\,\{\,m_\lambda:\lambda\in\RR\,\}\}\rangle$.

Not every abstract algebra $\langle V,\,\{0,\,+,\,\mathrm{ADDINV},\, \{\,m_\lambda:\lambda\in\RR\,\}    \}\rangle$, with a constant $0$, a binary operation +, a unary operation $\mathrm{ADDINV}$, and unary operations $m_\lambda$ for $\lambda\in\RR$ becomes a vector space when one attempts to introduce the vector space operations in the obvious way. For this, certain relations between the constants and the operations have to hold, such as $m_{\lambda_1\lambda_2}(x)=m_{\lambda_1}(m_{\lambda_2}(x))$ for $\lambda_1,\lambda_2\in\RR$, and $x\in\absalg$, and $x+(\mathrm{ADDINV}(x))=0$ for all $x\in \absalg$. This need not always be the case. In \cref{res:identities_forced}, below, it will become clear how  one can always pass to an abstract quotient algebra (to be defined below) of $\langle V,\,\{0,\,+,\,\mathrm{ADDINV},\,\{\,m_\lambda:\lambda\in\RR\,\}\}\rangle$ that \emph{is} a vector space.

\begin{definition}
	Let $\langle\absalg,\mathcal F^\absalg\rangle$ and $\langle\absalgtwo,\mathcal F^\absalgtwo\rangle$ be abstract algebras of the same type $\type$. Suppose that $h:\absalg\to\absalgtwo$ is a map. Then $h$ is an \emph{abstract algebra homomorphism} when the following are both satisfied:
	\begin{enumerate}
		\item $h(f^\absalg)=f^\absalgtwo$ for all $f\in \mathcal F$ such that $\rho(f)=0$;
		\item $h\big (f^\absalg(a_1,\dotsc,a_{\rho(f)})\big )=f^\absalgtwo\big (h(a_1),\dotsc,h(a_{\rho(f)})\big)$ for all $f\in\mathcal F$ such that $\rho(f)\geq 1$.
	\end{enumerate}
\end{definition}

The inclusion map from an abstract subalgebra to the abstract super-algebra is an abstract algebra homomorphism. The projections from an abstract product algebra to its factors are abstract algebra homomorphisms.

\begin{example}\label{ex:second_group_example}\quad
	\begin{enumerate}
		\item Let $G_1$ and $G_2$ be groups. The abstract algebra homomorphism $h:\langle G_1,\{e,{\,}^{-1}\,,\,\cdot\,\}\rangle\to\langle G_2,\{e,{\,}^{-1}\,,\,\cdot\,\}\rangle$ are maps between the underlying sets that are unital, preserve the inverse of one element, and preserve the product of two elements. Since $G_1$ and $G_2$ are actually groups, this is equivalent to preserving the product of two elements. Thus the abstract algebra homomorphisms between the associated abstract algebras are in a natural bijection with the group homomorphisms between the groups in the usual meaning of the word.
		\item Let $G_1$ and $G_2$ be groups. For $G_2$, take operations as in the second part of \cref{ex:first_group_example}. In that case, the abstract algebra homomorphisms $h:\langle G_1,\{e,{\,}^{-1}\,,\,\cdot\,\}\rangle\to\langle G_2,\{{\tilde f}_1^{G_2},{\tilde f}_2^{G_2},{\tilde f}_3^{G_2} \}\rangle$ are the maps $h:G_1\to G_2$ such that $h(e)=x_0$, $h(x^{-1})=x^7$ for $x\in G_1$, and $h(xy)=	x^2 y x^{-1} y^3 x^2$ for $x,y\in G_1$. Besides not being obviously natural or useful, it may well be the case that, for certain combinations of $G_1$, $G_2$, and $x_0$, such abstract algebra homomorphisms $h$ do not exist.
		\item Let $V$ and $W$ be vector spaces, and let  $\langle V,\,\{0^V,\,+^V,\,\mathrm{ADDINV}^V,  \, \{\,m_\lambda^V:\lambda\in\RR\,\} \}\rangle$ and $\langle W,\,\{0^W,\,+^V,\,\mathrm{ADDINV}^W, \, \{\,m_\lambda^W:\lambda\in\RR\,\} \}\rangle$ denote the naturally associated abstract algebras of the same unspecified type $\type$. Then the abstract algebra homomorphisms between the associated abstract algebras are in a natural bijection with the linear maps between the vector spaces.
	\end{enumerate}	
\end{example}

\begin{definition}
	Let $\absalg$ and $\absalgtwo$ be abstract algebras, and let $h:\absalg\to\absalgtwo$ be a map. Then the \emph{kernel of h}, denoted by $\ker h$, is defined as
	\[
	\ker h\coloneqq \left\{(x,y)\in \absalg^2: h(x)=h(y)\,\right\}.
	\]
\end{definition}

Note that $\ker h$ is not a subset of $\absalg$. In many practical contexts, however, it can be described in terms of a subset of $\absalg$ that will then be called the kernel of $h$ in the pertinent context. For example, let $G_1$ and $G_2$ be groups, and let  $h:\langle G_1,\{e,{\,}^{-1}\,,\,\cdot\,\}\rangle\to\langle G_2,\{e,{\,}^{-1}\,,\,\cdot\,\}\rangle$ be an abstract algebra homomorphism. Set $N\coloneqq \{\,x\in G_1: h(x)=e\,\}$. Then $\ker h=\{\,(x,y)\in G_1^2 : xy^{-1}\in N\,\}$. As another example, let $V$ and $W$ be vector spaces, and let
$h:V\to W$ be an abstract algebra homomorphism between the two associated abstract algebras. Set $L\coloneqq\{\,x\in V: h(x)=0\,\}$. Then $\ker h = \{\,(x,y)\in V^2 : x-y\in L\,\}$.

When $\theta\subseteq \absalg^2$ is a binary relation on $\absalg$, and $x,y\in \absalg$, then we shall write $x\,\theta\,y$ for $(x,y)\in\theta$. The kernels of abstract algebra homomorphisms turn out to be precisely the binary relations on $\absalg$ that we shall now define.

\begin{definition}
	Let $\langle\absalg,\mathcal F^\absalg\rangle$ be an abstract algebra of type $\type$. Then a binary relation $\theta\subseteq \absalg^2$ on $\absalg$ is called a \emph{congruence relation on $\absalg$} when the following are both satisfied:
	\begin{enumerate}
		\item $\theta$ is an equivalence relation on $\absalg$;
		\item when $f\in\mathcal F$ is such that $\rho(f)\geq 1$, then
		\[
		f^\absalg(x_1,\dotsc,x_{\rho(f)})\,\,\theta\,\,f^\absalg(y_1,\dotsc,y_{\rho(f)})
		\]
		whenever $x_1,\dotsc,x_{\rho(f)}\in \absalg$ and $y_1,\dotsc,y_{\rho(f)}\in \absalg$ are such that $x_i\,\theta\,y_i$ for $i=1,\dotsc,\rho(f)$.		
	\end{enumerate}
\end{definition}

It is clear that $\absalg^2$ is the largest congruence relation on $\absalg$, and that $\{\,(x,x): x\in\absalg\,\}$ is the smallest. The intersection of an arbitrary non-empty collection of congruence relations on $\absalg$ is again a congruence relation on $\absalg$. Suppose that $S\subseteq\absalg^2$ is an arbitrary subset. The intersection of all congruence relations on $\absalg$ that contain $S$ is the smallest congruence relation on
$\absalg$ that contains $S$; it is called the \emph{congruence relation on $\absalg$ that is generated by $S$}.

It is immediate from the definitions that the kernel of an abstract algebra homomorphism between abstract algebras of the same type is a congruence relation on the domain. All congruence relations on an abstract algebra occur in this fashion, as will become clear from the following construction of abstract quotient algebras.

Let $\absalg$ be an abstract algebra of type $\langle\mathcal F,\rho\rangle$. Suppose that $\theta$ is a congruence relation on $\absalg$. Let $\absalg/\theta$ denote the set of equivalence classes in $\absalg$ with respect to $\theta$, and let $q_\theta: \absalg\to\absalg/\theta$ denote the canonical map. When $f\in \mathcal F$ is such that $\rho(f)=0$, we set
\[
f^{\absalg/\theta}\coloneqq q_\theta(f^\absalg).
\]
When $f\in \mathcal F$ is such that $\rho(f)\geq 1$, then, for $x_1,\dotsc,x_{\rho(f)}\in\absalg$, we set
\[
f^{\absalg/\theta}\big(q_\theta(x_1),\dotsc,q_\theta(x_{\rho(f)})\big)\coloneqq q_\theta\big(f^\absalg(x_1,\dotsc,x_{\rho(f)})\big).
\]
Since $\theta$ is a congruence relation on $\absalg$, the maps $f^{\absalg/\theta}$ are well defined. Thus $\langle\absalg/\theta,\desset{f^{\absalg/\theta}:f\in \mathcal F}\rangle$ is an abstract algebra of type $\type$. By its construction, the map $q_\theta:\absalg\to\absalg/\theta$ is an abstract algebra homomorphism, and $\ker q_\theta=\theta$.

The following result, which is \cite[Exercise~1.26.8]{bergman_UNIVERSAL_ALGEBRA:2012}, is an immediate consequence of the definitions.

\begin{lemma}\label{res:factoring_a_homomorphism}
	Let $\absalg$ and $\absalgtwo$ be abstract algebras of the same type, and let $h:A\to B$ be an abstract algebra homomorphism. Suppose that $\theta$ is a congruence relation on $\absalg$ such that $\theta\subseteqq\ker h$. Then there exists a unique map $\fact{h}:\absalg/\theta\to\absalgtwo$ such that $h=\fact{h}\circ q_\theta$, and this map $\fact{h}$ is an abstract algebra homomorphism.
\end{lemma}

We now come to the main point of this section. Let $\type$ be a type. The abstract algebras of type $\type$, together with the abstract algebra homomorphisms between them, form a subcategory $\ALGRHO$ of $\SET$. Take a non-empty set $S$. Does there exists a free abstract algebra of type $\type$ over $S$? The answer is affirmative, and we shall now construct such an object $\FSETALGRHO{S}$ of $\ALGRHO$. It will be denoted by $\termalgS$.

The idea is quite easy. One starts by defining a set $\termalgS$ of words that reflect the concept of applying maps (symbolised by the elements of $\mathcal F$) to their appropriate numbers of variables (as prescribed by $\rho$), and keep repeating combining the outcomes to get new maps of an ever increasing degree of complexity (measured by what will be called the `height', below). The symbols in these words that reflect the concept of variables are taken from $S$. This set of words is then made into an abstract algebra of type $(\mathcal F, \rho)$ in a natural way, with concatenation reflecting the concept of combining outcomes of operations as input for another operation. Furthermore, when $\absalg$ is any abstract algebra of type $\type$, and $h:S\to \absalg$ is any map, then there is a natural abstract algebra homomorphism $\fact{h}$ from $\termalgS$ into $\absalg$ that extends $h$. This map $\fact{h}$ simply replaces each symbol $f$ from $\mathcal F$ in a word in $\termalgS$ by the concrete operation (or constant) $f^\absalg$ in the context of $\absalg$, and replaces each symbol $s$ from $S$ by the concrete element $h(s)$ of $\absalg$. All in all, this map $\fact{h}$ takes in a word from $\termalgS$ and then applies the `actual map that the word stands for in the context of $\absalg$' to values of its arguments that are the pertinent given elements $h(s)$ of $\absalg$.

The details are as follows; they are taken from \cite[p.95-96]{bergman_UNIVERSAL_ALGEBRA:2012} (where the case where  $S=\emptyset$ but $\{\,f\in\mathcal F:\rho(f)=0\,\}\neq\emptyset$ is also included). The structure of the proof of \cref{res:term_algebra_is_free}, which we include for the convenience of the reader, is also taken from that source; cf.~\cite[Theorem~ 4.32]{bergman_UNIVERSAL_ALGEBRA:2012}.

\begin{definition}\label{def:termalg}
	Let $\type$ be a type. Let $S$ be a non-empty (possibly infinite) set that is disjoint from $\mathcal F$. We recursively define a set of words in symbols $f$ for $f\in\mathcal F$ and symbols $s$ for $s\in S$, as follows. We set
	\[
	T_0(S)\coloneqq\{\,s: s\in S\,\}\cup\{f\in\mathcal F: \rho(f)=0\}\}
	\]
	and, for $n=1,2,\dotsc$, we set
	\[
	T_{n+1}(S)\coloneqq T_n(S)\cup\{\, f t_1\ldots t_{\rho(f)} : f\in\mathcal F, \,\rho(f)\geq 1,\, t_{{\phantom(}\!\!1},\dotsc,t_{\rho(f)}\in T_n(S)\,\}.
	\]
	We define $\termalgS\coloneqq\bigcup_{n\geq 0}T_n(S)$, and refer to elements of $\termalgS$ as \emph{terms of type $\type$ over $S$}. For a term $t\in\termalgS$, the smallest $n$ such that $t\in T_n(S)$ is called the \emph{height} of $t$.
\end{definition}

The terms of height zero, i.e., the elements of $T_0(S)$, can come in two kinds. Since $S$ is non-empty, there are always terms in $T_0(S)$ that consist of a single symbol $s$ from $S$. These can be thought of as `variables'. The other words in $T_0(S)$ are the symbols $f$ from $\mathcal F$ such that $\rho(f)=0$. There need not be any such $f$, but when there are, then the corresponding terms can be thought of as `constants' or, perhaps even better with an eye towards applications, as `distinguished elements'.

\begin{example}
	By way of (a rather finite) example, we consider the case where $\mathcal F=\{f_0,f_1,f_2\}$, $\rho(f_0)=0$, $\rho(f_1)=1$, $\rho(f_2)=2$, and where $S=\{x,y,z\}$. Then the set $T_0(S)$ of terms of height 0 consists of the `variables' $x$, $y$, and $z$, together with the `constant' $f_0$. The set $T_1(S)$ consists of all terms in $T_0(S)$ (all of height 0); the terms (all of height 1) $f_1 z$, $f_1 y$, $f_1 z$, and $f_1f_0$; and $4\cdot 4=16$ terms  (all of height 1) of the form $f_2\xi_1 \xi_2$, where each of $\xi_1,\xi_2$ can be taken from $T_0(S)=\{x,y,z,f_0\}$. The term $f_2xy$ of height 1, secretly translated into $f_2(x,y)$, reflects the concept of applying an operation that depends on two variables. The term $f_2xf_0$, translated into $f_2(x,f_0)$, reflects the concept of applying a map that depends on two variables with the second one fixed at the value $f_0$. The subset $T_2(S)$ of $\termalgS$ contains terms such as $f_2f_2 f_0yf_1 z$ and $f_2 f_1 x f_2 yx$. After translating these terms of height 2 into their more readable forms $f_2(f_2(f_0,y),f_1(z))$ and $f_2(f_1(x),f_2(y,x))$, respectively, it becomes clear which concepts they reflect.
\end{example}

We shall now supply the set $\termalgS$ with the structure of an abstract algebra of type $\type$. If $f\in \mathcal F$ is such that $\rho(f)=0$, then we set
\begin{equation}\label{eq:term_algebra_constants}
f^\termalgS\coloneqq f,
\end{equation}
and when $f\in\mathcal F$ is such that $\rho(f)\geq 1$, then we use concatenation of words to set
\begin{equation}\label{eq:term_algebra_other_maps}
f^\termalgS (t_1,\dotsc,t_{\rho(f)})\coloneqq f t_1\ldots t_{\rho(f)}
\end{equation}
for all $t_1,\ldots,t_{\rho(f)}\in\termalgS$.

\begin{theorem}\label{res:term_algebra_is_free}
	Let $\type$ be a type. For every abstract algebra $\absalg$ of type $\type$ and every map $h:S\to \absalg$, there is a unique abstract algebra homomorphism $\fact{h}:\termalgS\to\absalg$ such that $\fact{h}(s)=h(s)$ for all $s\in S$:
	\begin{equation*}
	\begin{tikzcd}
	S\arrow[Subseteq]{r}{}\arrow[dr, "h"]& \termalgS\arrow[d, "\fact{h}"]
	\\& A
	\end{tikzcd}
	\end{equation*}
	That is, $\FSETALGRHO{S}$ exists and is equal to $\termalgS$; the accompanying inclusion map $j$ is injective.
	
\end{theorem}

\begin{proof}
	The construction of the map $\fact{h}$ and the proof of its uniqueness can be given simultaneously, using induction on the height of a term.
	
	Take a term $t\in\termalgS$ of height 0. If $t$ is a word that consists of a single symbol from $S$, then $\fact{h}(s)$ is prescribed, and we define $\fact{h}(s)\coloneqq h(s)$ accordingly. If $t$ is a symbol $f$ from $\mathcal F$ such that $\rho(f)=0$, then, since $\fact{h}$ is supposed to be an abstract algebra homomorphism, \cref{eq:term_algebra_constants} implies that we must have $\fact{h}(f)=h(f^\termalgS)=f^\absalg$. Hence we define $\fact{h}(f)\coloneqq f^\absalg$ accordingly.

	Suppose that, for some $n\geq 0$, the uniqueness of $\fact{h}(t)$ has already been shown for all terms $t\in\termalgS$ of height at most $n$, and that $\fact{h}(t)$ has already been defined accordingly for such $t$. Take a term $t$ of height $n+1$. Then $t=f t_1\dotsc t_{\rho(f)}$ for a unique $f\in\mathcal F$ such that $\rho(f)\geq 1$ and unique terms $t_1,\dotsc,t_{\rho(f)}\in\termalgS$ of height at most $n$. Since $\fact{h}$ is supposed to be an abstract algebra homomorphism, \cref{eq:term_algebra_other_maps} implies that we must have
	\begin{align*}
	\fact{h}(t)&=\fact{h}(f t_1\dotsc t_{\rho(f)})\\
	&=\fact{h}(f^\termalgS(t_1,\dotsc,t_{\rho(f)}))\\
	&=f^\absalg(\fact{h}(t_1),\dotsc,\fact{h}(t_{\rho(f)})).
	\end{align*}
	As a consequence of the induction hypotheses, this shows that $\fact{h}(t)$ is also uniquely determined. Since, also as a consequence of the induction hypothesis, $\fact{h}(t_1),\ldots,\fact{h}(t_{\rho(f)})$ have already been defined, we can now define
	\[
	\fact{h}(t)\coloneqq f^\absalg(\fact{h}(t_1),\dotsc,\fact{h}(t_{\rho(f)}))
	\]
	accordingly.
	This completes the induction step.
	
	We have now shown that $\fact{h}$ is uniquely determined as well as explicitly constructed the only possible candidate. It is immediate from this construction and the definitions in  \cref{eq:term_algebra_constants,eq:term_algebra_other_maps} that this candidate is indeed an abstract algebra homomorphism.
	
	The final sentence of the statement is then clear.
\end{proof}

\begin{remark}\label{rem:term_algebra_is_generated_by_S}
	It is evident from its construction that $\termalgS$ is generated by $S$, in the sense that it equals its smallest abstract subalgebra that contains $S$. Of course, \cref{rem:generating} also makes clear that this must be the case.
\end{remark}

\section{Various lattices as abstract algebras satisfying identities}\label{sec:various_lattices_as_abstract_algebras_satisfying_identities}

\noindent Before proceeding with the general theory from universal algebra that we need, we pause to discuss structures that involve a partial ordering.

It is clear that, for a vector space, the validity of the vector space axioms can be expressed in terms of identities that involve the constant 0 and the operations of the naturally associated abstract algebra. For a unital vector lattice algebra, for example, it is, however, far less clear that there is an associated abstract algebra for which is possible. After all, the axioms of a unital vector lattice algebra also involve \emph{in}equalities and the assumption of the existence of the infimum and supremum of two elements. At first sight, it may seem counterintuitive that these can also be described in terms of constants, operations, and identities. Nevertheless, this is possible. As we shall see, it is precisely this fact that lies at the heart of the existence proof for the free unital vector lattice algebra over a non-empty set. Later on, we shall then use this one existence result to obtain the existence of all other fourteen free objects in \cref{res:overview_algebraic}, below.

We start with the classical observation that the partial ordering in a lattice can equivalently be formulated in terms of operations and identities; see \cite[Definition~1.7 and p.23]{bergman_UNIVERSAL_ALGEBRA:2012}, for example. Since this is crucial, and since we strive to keep this paper self-contained, we include the details for this. The ad-hoc terminology in the following definition is ours. We refrain from claiming any other originality here.

\begin{definition}\label{def:two_types_of_latices}
	Let $S$ be a non-empty set.
	\begin{enumerate}
		\item Suppose that $\leq$ is a partial ordering on $S$. Then the partially ordered set $(S,\leq)$ is a \emph{partially ordered lattice} if, for all $x,y\in S$, the supremum $x\vee y$ and the infimum $x\wedge y$ exist in $S$.
		\item Suppose that $S$ is supplied with binary operations $\alwedge$ and $\alvee$. Then the abstract algebra $(S,\alwedge,\alvee)$ is an \emph{algebraic lattice} if, for all $x,y,z\in S$,
		\begin{align*}
		\quad\quad\quad&x\alwedge\left(y\alwedge z\right)=\left(x \alwedge y\right)\alwedge z, & &x\alvee\left(y\alvee z\right)=\left(x \alvee y\right)\alvee z,\\
		&x\alwedge x = x, & &x\alvee x = x,\\
		&x\alwedge y=y\alwedge x, & & x\alvee y = y \alvee x,\\
		& x\alwedge\left(x\alvee y\right)=x, \quad \text{and} & & x\alvee\left(x\alwedge y\right)=x.
		\end{align*}
	\end{enumerate}
\end{definition}

Let us mention explicitly that distributivity is not supposed.

\begin{lemma}\label{res:two_types_of_lattices_are_equivalent}
	Let $S$ be a non-empty set.
	\begin{enumerate}
		\item Suppose that $S$ is supplied with a partial ordering such that the partially ordered set $(S,\leq)$ is a partially ordered lattice. For $x,y\in S$, set
		\begin{align*}
		x\alwedge y &\coloneqq x\wedge y
		\intertext{and}
		x\alvee y &\coloneqq x\vee y.
		\end{align*}
		Then the abstract algebra $(S,\alvee,\alwedge)$ is an algebraic lattice.
		\item Suppose that $S$ is supplied with two binary operations $\alwedge$ and $\alvee$ such that the algebra $(S,\alwedge,\alvee)$ is an algebraic lattice. For $x,y\in S$, say that $x\leq y$ if and only if
		\[
		x\alwedge y=x.
		\]
		Then $\leq$ is a partially ordering on $S$, and the partially ordered set $(S,\leq)$ is a partially ordered lattice. Moreover, for $x,y\in S$, we have
		\begin{align*}
		x\wedge y &= x\alwedge y
		\intertext{and}
		x\vee y &= x\alvee y,
		\end{align*}
		where $x\wedge y$ and $x\wedge y$ refer to the infimum and the supremum, respectively, in the partial ordering $\leq$.	
	\end{enumerate}
\end{lemma}

\begin{proof}
	It is completely routine to verify the statement in part (1).
	
	We turn to part (2). It is an easy consequence of the first three identities in the left column in \cref{def:two_types_of_latices} that $\leq$ is a partially ordering on $S$.
	
	Take $x,y\in S$. We claim that $x\wedge y$ exists in the partially ordered set $(S,\leq)$ and that, in fact, $x\wedge y =x\alwedge y$. Since $\left(x\alwedge y\right)\alwedge y = x\alwedge \left(y\alwedge y\right)=x\alwedge y$, we have $x\alwedge y\leq y$. Since $\left(x\alwedge y\right)\alwedge x = x\alwedge \left(x\alwedge y\right)=\left(x\alwedge x\right)\alwedge y=x\alwedge y$, we have $x\alwedge y\leq x$. Then also $x\alwedge y=y\alwedge x\leq x$. Take $z\in S$, and suppose that $z\leq x$ and $z\leq y$, i.e., suppose that $z\alwedge x=z$ and $z\alwedge y=y$. Then $z\alwedge (x\alwedge y)=(z\alwedge x)\alwedge y= z\alwedge y=z$. Hence $z\leq x\alwedge y$ and the proof of the claim is complete.
	
	Before turning to the supremum, we note that, for $x,y\in S$, the fact that $x\alwedge y=x$ is equivalent to the fact that $x\alvee y=y$. It is here that the two identities in the fourth line of identities in \cref{def:two_types_of_latices} come in. Indeed, suppose that $x\alwedge y=x$. Then $x\alvee y =\left(x\alwedge y\right)\alvee y=y\alvee\left(x\alwedge y\right)=y\alvee\left(y\alwedge x\right)=y$. Conversely, suppose that $x\alvee y=y$. Then $x\alwedge y=x\alwedge\left(x\alvee y\right)=x$.
	
	Now take $x,y\in S$. We claim that $x\vee y$ exists in the partially ordered set $(S,\leq)$ and that, in fact, $x\vee y =x\alvee y$. Since $x\alwedge \left(x\alvee y\right)=x$, we have $x\leq x\alvee y$. Since $y\alwedge\left(x\alvee y\right)= y\alwedge\left(y\alvee x\right)=y$, we have $y\leq x\alvee y$. Take $z\in S$ and suppose that $x\leq z$ and $y\leq z$, i.e., suppose that $x\alwedge z=x$ and $y\alwedge z=y$. By what we have just established in the intermezzo, this is equivalent to supposing that $x\alvee z=z$ and $y\alvee z=z$. Then $\left(x\alvee y\right)\alvee z=x\alvee\left(y\alvee z\right)=x\alvee z=z$. Again by the intermezzo, it follows that $\left(x\alvee y\right)\alwedge z=x\alvee y$. Hence $x\alvee y\leq z$ and the proof of the claim is complete. Part (2) has now been established.

\end{proof}

The following is now clear from \cref{res:two_types_of_lattices_are_equivalent}.

\begin{proposition}\label{res:LAT_is_an_equational_class}
	The constructions in the parts \uppars{1} and \uppars{2} of \cref{res:two_types_of_lattices_are_equivalent} yield mutually inverse bijections between the category of partially ordered lattices \uppars{with the lattice homomorphisms as morphisms} and the category of abstract algebras that are algebraic lattices \uppars{with the abstract algebra homomorphisms as morphisms}. Under this isomorphism, the underlying sets and the maps that are the morphisms are kept.
\end{proposition}

With \cref{res:two_types_of_lattices_are_equivalent} available, it is not so difficult to describe the unital vector lattice algebras as the abstract algebras (of a common unspecified type) where certain identities are satisfied.

\begin{lemma}\label{res:unital_vla_as_abstract_algebra}
	Let $\absalg$ be an abstract algebra with \uppars{not necessarily different} constants 0 and 1, a binary map $\oplus$, a unary map $\ominus$, a unary map $m_\lambda$ for every $\lambda\in\RR$, a binary map $\odot$, and binary maps $\alwedge$ and $\alvee$. Suppose that all of the following are satisfied:
	\begin{enumerate}
		\item $(x\oplus y)\oplus z=x\oplus(y\oplus z)$ for all $x,y,z\in \absalg$;\label{id:group_first}
		\item $x\oplus 0= x$ for all $x\in \absalg$;
		\item $x\oplus(\ominus x)=0$ for all $x\in \absalg$;
		\item $x\oplus y=y\oplus x$ for all $x,y\in \absalg$;\label{id:group_last}
		\item $m_\lambda(x\oplus y)=m_\lambda(x)\oplus m_\lambda(y)$ for all $\lambda\in\RR$ and $x,y\in \absalg$;\label{id:vector_space_first}
		\item $m_{\lambda + \mu}(x)=m_{\lambda}(x)\oplus m_{\mu}(x)$ for all $\lambda,\mu\in\RR$ and $x\in \absalg$;
		\item $m_{\lambda\mu}(x))=m_{\lambda}(m_\mu(x))$ for all $\lambda,\mu\in\RR$ and $x\in \absalg$;
		\item $m_1(x)=x$ for all $x\in \absalg$;\label{id:vector_space_last}
		\item $(x\odot y)\odot z=x\odot(y\odot z)$ for all $x,y,z\in \absalg$;\label{id:ass_alg_first}
		\item $x\odot(y\oplus z)=(x\odot y)\oplus(x\odot z)$ for all $x,y,z\in \absalg$;
		\item $(x\oplus y)\odot z=(x\odot z)\oplus(y\odot z)$ for all $x,y,z\in \absalg$;
		\item $m_\lambda(x\odot y)=m_\lambda(x)\odot y = x\odot m_\lambda(y)$ for all $\lambda\in\RR$ and $x,y\in \absalg$;
		\item $1\odot x=x\odot 1=x$ for all $x\in \absalg$;\label{id:ass_alg_last}
		\item $x\alwedge\left(y\alwedge z\right)=\left(x \alwedge y\right)\alwedge z$ and $x\alvee\left(y\alvee z\right)=\left(x \alvee y\right)\alvee z$ for all $x,y,z\in \absalg$;\label{id:lattice_first}
		\item $x\alwedge x = x$ and $x\alvee x = x$ for all $x\in \absalg$;
		\item $x\alwedge y=y\alwedge x$ and $x\alvee y = y \alvee x$ for all $x,y\in \absalg$;
		\item $x\alwedge\left(x\alvee y\right)=x$ and $x\alvee\left(x\alwedge y\right)=x$ for all $x,y\in \absalg$;\label{id:lattice_last}
		\item $x\oplus(y\alwedge z)=(x\oplus y)\alwedge(x\oplus z)$ for all $x,y,z\in \absalg$;\label{id:translation_and_ordering}
		\item $m_{\lambda}(0\alwedge x)=0\alwedge (m_\lambda (x))$ for all $\lambda\in\RR_{\geq 0}$ and $x\in \absalg$;\label{id:scalar_multiplication_and_ordering}
		\item $0 \alwedge ((x\alwedge(\ominus x))\odot(y\alwedge(\ominus y)))=0$ for all $x,y\in \absalg$.\label{id:algebra_multiplication_and_ordering}
	\end{enumerate}
	
	Set
	\begin{enumerate}
		\item[\textup{(a)}] $x+y\coloneqq x\oplus y$ for $x,y\in \absalg$;
		\item[\textup{(b)}] $\lambda x\coloneqq m_{\lambda}(x)$ for $\lambda\in\RR$ and $x\in \absalg$;
		\item[\textup{(c)}] $xy\coloneqq x\odot y$ for $x,y\in \absalg$.
	\end{enumerate}
	Supplied with the operations as defined under \textup{(a)}, \textup{(b)}, and \textup{(c)}, $\absalg$ is an associative algebra over the real numbers with zero element 0 and identity element 1.
	
	For $x,y\in \absalg$, say that $x\leq y$ when $x\alwedge y=x$. Then $\leq$ is a partial ordering on $\absalg$ that makes $\absalg$ into a partially ordered lattice. Moreover, for $x,y\in \absalg$, we have $x\wedge y=x\alwedge y$ and $x\vee y=x\alvee y$, where $\wedge$ and $\vee$ refer to the supremum resp.\ infimum in the partial ordering $\leq$.

	Supplied with the partial ordering $\leq$ and with the operations as defined under \textup{(a)}, \textup{(b)}, and \textup{(c)}, $\absalg$ is a unital vector lattice algebra with zero element 0 and identity element 1.
\end{lemma}

\begin{proof}
	The identities in \partref{id:group_first}--\partref{id:group_last} show that $\absalg$ becomes an abelian group under $+$ with identity element 0. Those in \partref{id:vector_space_first}--\partref{id:vector_space_last} show that $\absalg$ becomes a vector space with zero element 0 when the scalar multiplications as in (b) are added, and those in \partref{id:ass_alg_first}--\partref{id:ass_alg_last} guarantee that $\absalg$ becomes an associative algebra with zero element 0 and identity element 1 when the multiplication as in (c) is added.
	
	It becomes more interesting when the partial ordering is brought in. In view of \cref{res:two_types_of_lattices_are_equivalent}, the identities in \partref{id:lattice_first}--\partref{id:lattice_last} guarantee that $\leq$ is a partial ordering on $\absalg$ that makes $\absalg$ into a partially ordered lattice where the infimum $x\wedge y$ and supremum $x\vee y$ of two elements $x,y$ are given by $x\alwedge y$ and $x\alvee y$, respectively.
	
	It remains to be shown that the partial ordering $\leq$ is a vector space ordering, and also that the product of two positive elements of $\absalg$ is again positive.
	
	We start with the vector space ordering. Take $y,z\in \absalg$ and suppose that $y\leq z$, i.e, suppose that $y\alwedge z=y$. Take $x\in \absalg$. Using the identity in \partref{id:translation_and_ordering}, we have
	\[
	(x+y)\wedge (x+z)=(x\oplus y)\alwedge(x\oplus z)=x\oplus(y\alwedge z)=x\oplus y=x +y.
	\]
	Hence $x+z\leq y+z$. Take $x\in \absalg$ and $\lambda\in\RR_{\geq 0}$. Suppose that $0\leq x$, i.e, suppose that $0\alwedge x=0$. Using the identity in \partref{id:scalar_multiplication_and_ordering} (and, in the final equality,  the fact that we already know that $\absalg$ is a vector space), we have
	\[
	0\wedge(\lambda x)=0\alwedge(m_\lambda(x))=m_\lambda(0\alwedge x)=m_\lambda(0)=\lambda 0=0.
	\]
	Hence $0\leq \lambda x$,
	
	We turn to the product of two positive elements of $\absalg$. Since we know by now that $\absalg$ is a vector lattice, we can equivalently formulate the identity in \partref{id:algebra_multiplication_and_ordering} as the fact that $0\wedge (\abs{x}\abs{y})=0$ for all $x,y\in \absalg$. This implies (and is equivalent to) the fact that the product of two positive elements of $\absalg$ is again positive.
\end{proof}

Obviously, any unital vector lattice algebra gives rise, in a natural way, to an abstract algebra with constants and operations as in \cref{res:unital_vla_as_abstract_algebra} where all identities in (1)--(20) in \cref{res:unital_vla_as_abstract_algebra} are satisfied. As for partially ordered lattices and algebraic lattices, the two constructions are mutually inverse. Moreover, the unital vector lattice algebra homomorphisms correspond to the abstract algebra homomorphisms. We therefore have the following analogue of \cref{res:LAT_is_an_equational_class}.

\begin{proposition}\label{res:VLAONE_is_equational_class}
	The category $\VLAONE$ of unital vector lattice algebras, with the unital vector lattice algebra homomorphisms as morphism, is isomorphic to the category of abstract algebras with constants and operations as in \cref{res:unital_vla_as_abstract_algebra} where the identities \uppars{1}--\uppars{20} in \cref{res:unital_vla_as_abstract_algebra} are satisfied, with the abstract algebra homomorphisms as morphisms. Under this isomorphism, the underlying sets and the maps that are the morphisms are kept.
\end{proposition}

Obviously, there are isomorphisms similar to those in \cref{res:LAT_is_an_equational_class,res:VLAONE_is_equational_class} for $\VL$ and $\VLA$.  Once one notices that one can express the positivity of an identity element 1 of a vector lattice algebra by requiring that $0\alwedge 1=0$, it becomes clear that there is also a similar isomorphism for $\VLAONEPOS$. For many categories of algebraic structures (groups, abelian groups, vector spaces, rings with identity elements, algebras, commutative algebras, \ldots), where there is no partial ordering that needs to be `equationalised', the existence of a similar isomorphism with a category of abstract algebras is immediate from the axioms for these structures.

The existence of a `similar' isomorphism can be made precise by saying that all these categories are isomorphic to an equational class of abstract algebras. This brings us to the next section.

\section{Universal algebra: part II}\label{sec:universal_algebra_part_II}

\noindent We have seen in \cref{res:LAT_is_an_equational_class} and \cref{res:VLAONE_is_equational_class} how two categories from the partially ordered realm are isomorphic to categories of abstract algebras. In the abstract algebraic side of the picture, the objects of the category are those abstract algebras (all of a common type) `where certain identities are satisfied'. We shall now formalise this concept of `identities being satisfied'.

Let $\absalg$ be an abstract algebra of type $\type$. By way of example, suppose that one of the operations on $\absalg$ is a binary operator $\oplus$. We want to express the fact that this operation is associative, i.e., that $x_1\oplus(x_2\oplus x_3)=(x_1\oplus x_2)\oplus x_3$ for all $x_1,x_2,x_3\in A$. A first attempt would be to take a set $S=\{s_1,s_2,s_3\}$ of three elements, take the terms (rewritten in a legible way) $(s_1\oplus s_2)\oplus s_3$ and $s_1\oplus(s_2\oplus s_3)$ in $\termalgS$, and require that $(s_1\oplus s_2)\oplus s_3= s_1\oplus(s_2\oplus s_3)$ `is satisfied in $\absalg$'. The problem is that this does not make sense. Firstly, the left and the right hand sides are not elements of $\absalg$. They are elements of $\termalgS$ and, secondly, they are \emph{not} equal in $\termalgS$. There are two ways to get further.

The first one is to take the terms $(s_1\oplus s_2)\oplus s_3$ and $s_1\oplus(s_2\oplus s_3)$, and assign to them the ternary operations on $\absalg$ that send a triple $(x_1,x_2,x_3)\in\absalg^3$ to $x_1\oplus(x_2\oplus x_3)$ and $(x_1\oplus x_2)\oplus x_3$, respectively. The associativity can then be expressed by saying that these two maps from $\absalg^3$ to $\absalg$ are equal. This is the approach that is taken for the general case in \cite[Definitions~4.31 and~4.35 ]{bergman_UNIVERSAL_ALGEBRA:2012}. Here, in order to be able to accommodate identities in an arbitrarily large number of variables, one takes a countably infinite set $S$. Given two terms $t_1,t_2\in\termalgS$, one associates operations $t_1^\absalg$ and $t_2^\absalg$  with them, and requires that these be equal as maps from $\absalg^n$ (where $n$ is appropriate) to $\absalg$. There are some formalities to be taken care of then, however. For example, it could be the case that the `natural' number of arguments of $t_1^A$ differs from that of $t_2^A$. One could want to express the fact that $x_1\oplus x_2 = (x_3\oplus x_2)\oplus x_1$ for all $x_1,x_2,x_3\in \absalg$, but the natural domain for the map for the left hand side is $\absalg^2$, whilst for the right hand side this is $\absalg^3$.

The second one, and the one we shall take, is the following. Take a three-point set $S$ again. For all $x_1,x_2,x_3\in \absalg$, there exists a map $h_{x_1,x_2,x_3}:S\to \absalg$ such that $h_{x_1,x_2,x_3}(s_i)=x_i$ for $i=1,2,3$. By \cref{res:term_algebra_is_free}, such a map extends uniquely to an abstract algebra homomorphism $\fact{h}_{x_1,x_2,x_3}:\termalgS\to\absalg$. Then $\fact{h}_{x_1,x_2,x_3}(s_1\oplus(s_2\oplus s_3))=x_1\oplus(x_2\oplus x_3)$ and $\fact{h}_{x_1,x_2,x_3}((s_1\oplus s_2)\oplus s_3)=(x_1\oplus x_2)\oplus x_3$. It follows from this that the associativity of $\oplus$ in $\absalg$ can equally well be expressed by requiring that $\fact{h}(s_1\oplus(s_2\oplus s_3))=\fact{h}((s_1\oplus s_2)\oplus s_3)$ for every map $h: S\to\absalg$, where, as usual, $\fact{h}:\termalgS\to\absalg$ is the abstract algebra homomorphism that extends $h$. Since every abstract algebra homomorphism from $\termalgS$ to $\absalg$ is the unique extension of its restriction to $S$, one can equally well (with a change in notation) require that $h(s_1\oplus(s_2\oplus s_3))=h((s_1\oplus s_2)\oplus s_3)$ for every abstract algebra homomorphism $h: \termalgS\to\absalg$. Of course, one wants to be able to do this with an arbitrarily large number of variables involved. This leads to the following definition, as in \cite[Definition~9.4.1]{bergman_AN_INVITAION_TO_GENERAL_ALGEBRA_AND_UNIVERSAL_CONSTRUCTIONS} and \cite[Section~2.8]{jacobson_BASIC_ALGEBRA_II_SECOND_EDITION:1989}

Fix, once and for all, a countably infinite set $\infset=\{s_1,s_2,\dotsc\}$.

\begin{definition}
	Let $\type$ be a type. Take two terms $t_1,t_2\in\termalginfset$. Let $\absalg$ be an abstract algebra of type $\type$. Then \emph{$\absalg$ satisfies $t_1\approx t_2$} when $h(t_1)=h(t_2)$ for every abstract algebra homomorphism $h:\termalginfset\to\absalg$. For a subset $\Sigma$ of $\termalginfset\times\termalginfset$,  \emph{$\absalg$ satisfies $\Sigma$} when $\absalg$ satisfies $t_1\approx t_2$ for every pair $(t_1,t_2)\in\Sigma$.
\end{definition}

This definition depends on the choice of $\infset$ because $t_1$ and $t_2$ are elements of the set $\termalginfset$ that depends on this choice. A moment's thought shows that the equalities of operations on $\absalg$\textemdash which is what we are after\textemdash that is equivalent to the satisfaction of the identities from $\Sigma$ is independent of this choice. It is for this reason that there is no harm in fixing a particular choice for $\infset$ as we have done.

For $\Sigma\subseteq\termalginfset\times\termalginfset$, the class of all abstract algebras of type $\type$ satisfying $\Sigma$ is called the \emph{equational class defined by $\Sigma$}. Together with the abstract algebra homomorphism between them, it forms the subcategory $\ALGRHOSIGMA$ of $\ALGRHO$.

It is an important point that we can force identities to be satisfied by passing to an abstract quotient algebra.

\begin{lemma}\label{res:identities_forced}
	Let $\absalg$ be an abstract algebra of type $\type$, and let $\theta$ be a congruence relation on $\absalg$. Take $t_1,t_2\in\termalginfset$. Then $\absalg/\theta$ satisfies $t_1\approx t_2$ if and only if  $(h(t_1),h(t_2))\in\theta$ for all abstract algebra homomorphisms $h:\termalginfset\to\absalg$.
\end{lemma}

\begin{proof}
	Let $q_\theta:A\to A/\theta$ be the quotient map. It is a consequence of the surjectivity of the abstract algebra homomorphism $q_\theta$ and the universal property of $\termalginfset$ that the abstract algebra homomorphisms $h_\theta:\termalginfset\to\absalg/\theta$ are precisely the compositions $q_\theta\circ h$ for the abstract algebra homomorphisms $h: \termalginfset\to\absalg$. The statement in the lemma is an immediate consequence of this.
\end{proof}

\begin{lemma}\label{res:kernel_contains_pairs}
	Let $\absalg$ and $\absalgtwo$ be abstract algebras of the same type $\type$, and let $h:\absalg\to\absalgtwo$ be an abstract algebra homomorphism. Take $t_1,t_2\in\termalginfset$. If $\absalgtwo$ satisfies $t_1\approx t_2$, then $(h^\prime(t_1),h^\prime(t_2))\in\ker h$ for every abstract algebra homomorphism $h^\prime:\termalginfset\to\absalg$.
\end{lemma}

\begin{proof}
	Take an abstract algebra homomorphism $h^\prime:\termalginfset\to\absalg$. Then $h\circ h^\prime:\termalginfset\to \absalgtwo$ is an abstract algebra homomorphism. Hence $(h\circ h^\prime)(t_1)=(h\circ h^\prime)(t_2)$, showing that $(h^\prime(t_1),h^\prime(t_2))\in\ker h$.
\end{proof}

We can now construct a free object of an equational class over a non-empty set. Let $\type$ be a type, and take a subset  $\Sigma\subseteqq\termalginfset\times\termalginfset$. Let $S$ be a non-empty set. We then take the abstract term algebra $\termalgS$, which contains $S$ as a subset, and we let $\theta$ be the smallest congruence relation on $\termalgS$ that contains the pairs $(h^\prime(t_1),h^\prime(t_2))$ for all $(t_1,t_2)\in\Sigma$ and all abstract algebra homomorphisms $h^\prime:\termalginfset\to\termalgS$. Let $q_\theta:\termalgS\to\termalgS/\theta$ denote the quotient map, and let $q_\theta|_S$ denote its restriction to $S$.

We see from \cref{res:identities_forced} that $\termalgS/\theta$ satisfies $\Sigma$, i.e., it is an object of $\ALGRHOSIGMA$.

Furthermore, we claim that the pair $(q_\theta|_S,\termalgS/\theta)$ is a free object of $\ALGRHOSIGMA$ over the object $S$ in $\SET$. To see this, let $\absalg\in\ALGRHOSIGMA$, and let $h:S\to\absalg$ be a map. By \cref{res:term_algebra_is_free}, there exists a unique abstract algebra homomorphism $\fact{h}:\termalgS\to\absalg$ such that $\fact{h}(s)=h(s)$ for all $h\in S$.

Take $(t_1,t_2)\in\Sigma$, and let $h^\prime:\termalginfset\to \termalgS$ be an arbitrary abstract algebra homomorphism.  Since $\absalg$ satisfies $t_1\approx t_2$, \cref{res:kernel_contains_pairs} shows that $(h^\prime(t_1),h^\prime(t_2))\in\ker \fact{h}$. Thus the congruence relation $\ker f$ on $\termalgS$ contains the generators of $\theta$, and we conclude that $\theta\subseteq\ker \fact{h}$. \cref{res:factoring_a_homomorphism} then implies that there is a unique abstract algebra homomorphism $\factfact{h}:\termalgS/\theta\to A$ such that $\fact{h}=\factfact{h}\circ q_\theta$. For $s\in S$, this implies that $(\factfact{h}\circ q_\theta|_S)(s)=\fact{h}(s)=h(s)$. Hence we have found a factoring abstract algebra homomorphism $\factfact{h}$. It remains to show uniqueness. Suppose that $h_0:\termalgS/\theta\to\absalg$ is an abstract algebra homomorphism such that $(h_0\circ q_\theta|_S) (s)=h(s)$ for all $s\in S$. Then $h_0\circ q_\theta: \termalgS\to A$ is an abstract algebra homomorphism such that $(h_0\circ q_\theta)(s)=(h_0\circ q_\theta|_S)(s)=h(s)=\fact{h}(s)$ for $s\in S$. This implies that $h_0\circ q_\theta=\fact{h}$ and this, in turn, shows that $h_0=\factfact{h}$.

All in all, we have shown the following. Its main part is the existence of a object, but we have also included the construction of a concrete realisation of it.

\begin{theorem}\label{res:free_algebra_with_relations_exists}
	Let $\type$ be a type, and take $\Sigma\subseteq\termalginfset\times\termalginfset$. Let $S$ be a non-empty set, and let $\theta$ be the smallest congruence relation on $\termalgS$ that contains the pairs $(h^\prime(t_1),h^\prime(t_2))$ for all $(t_1,t_2)\in\Sigma$ and all abstract algebra homomorphisms $h^\prime:\termalginfset\to\termalgS$. Then $\termalginfset/\theta$ is an abstract algebra of type $\type$ that satisfies $\Sigma$. Let $q_\theta:\termalgS\to\termalgS/\theta$ denote the quotient map, and let $q_\theta|_S$ denote its restriction to the subset $S$ of $\termalgS$.
	
	For every abstract algebra $\absalg$ of type $\type$ that satisfies $\Sigma$, and for every map $h:S\to \absalg$, there is a unique abstract algebra homomorphism $\fact{h}:\termalgS/\theta\to\absalg$ such that $h=\fact{h}\circ q_\theta|_S$ for all $s\in S$:
	
	\begin{equation*}
	\begin{tikzcd}
	S\arrow[r,"q_\theta|_S"]\arrow[dr, "h"]& \termalgS/\theta\arrow[d, "\fact{h}"]
	\\& \absalg
	\end{tikzcd}
	\end{equation*}
	
	That is, $\FSETALGRHOSIGMA{S}$ exists and is equal to $\termalgS/\theta$. 		
\end{theorem}

\begin{remark}
	\quad
	\begin{enumerate}
		\item If $\ALGRHOSIGMA$ contains an object that has at least two elements, then $q_\theta|_S$ must be injective. Consequently, $j$ is \emph{not} injective if and only if $S$ contains at least two elements and $\ALGRHOSIGMA$ consists only of the one-point algebra of type $\type$. For many equational classes of practical interest, $q_\theta|_S$ is, therefore, injective.
		\item For the choice $\Sigma=\emptyset$ one retrieves the fact from \cref{res:term_algebra_is_free} that $\FSETALGRHO{S}$ exists and equals $\termalgS$. The injectivity of the accompanying map $j$ is not obtained in this fashion, but follows easily from the universal property since $\termalgS$ is an abstract algebra of type $\type$ that has at least two elements.
	\end{enumerate}
\end{remark}	

\cref{res:free_algebra_with_relations_exists} is a classical result; see \cite[Theorem~2.10]{jacobson_BASIC_ALGEBRA_II_SECOND_EDITION:1989}, for example. It can also be found as (a part of)  \cite[Corollary~4.30]{bergman_UNIVERSAL_ALGEBRA:2012}, where it is actually proved\textemdash with a slightly different proof\textemdash that free objects over non-empty sets exist in varieties of abstract algebras. As in \cite[p.9]{bergman_UNIVERSAL_ALGEBRA:2012}, we say that a \emph{variety of abstract algebras} is a class of algebras of the same type that is closed under taking subalgebras, abstract algebra homomorphic images, and abstract algebra products. Some sources use both the terms `variety' and `equational class' for what we call an `equational class'; see \cite[p.81]{jacobson_BASIC_ALGEBRA_II_SECOND_EDITION:1989} and \cite[p.152]{gratzer_UNIVERSAL_ALGEBRA_SECOND_EDITION:1979}. As we shall see in a moment, a theorem of Birkhoff's shows that there is no actual ambiguity in the terminology.

It is routine to verify that $\ALGRHOSIGMA$ is closed under the taking of abstract subalgebras and the formation of arbitrary abstract algebra products. Using the universal property of $\termalginfset$, one sees that it is also closed under the taking of abstract algebra homomorphic images. Every equational class (in our terminology) is, therefore, a variety (in our terminology).
The converse is actually also true by Birkhoff's theorem; see \cite[Theorem~4.41]{bergman_UNIVERSAL_ALGEBRA:2012} and \cite[Theorem~2.15]{jacobson_BASIC_ALGEBRA_II_SECOND_EDITION:1989}, for example. Therefore, there is no ambiguity in terminology, and \cite[Corollary~4.30]{bergman_UNIVERSAL_ALGEBRA:2012} and \cref{res:free_algebra_with_relations_exists} are actually equivalent.

\section{Free vector lattices and free objects of categories of vector lattice algebras}\label{sec:free_vector_lattices_and_free_vector_lattice_algebras}

\noindent It is clear from \cref{res:VLAONE_is_equational_class} that the category of unital vector lattice algebras is isomorphic to a category of abstract algebras that is an equational class. At the level of objects, the isomorphism keeps the underlying sets, but views them as different structures. At the level of morphisms, the maps between the underlying sets are kept, but are observed to have different pertinent properties.  \cref{res:free_algebra_with_relations_exists} therefore implies that free unital vector lattice algebras over non-empty sets exist. Let us spell out the details once more. The additional properties under (1), (2), and (3) follow from general principles.

\begin{theorem}\label{res:FSETVLAONE_exists}
	Let $S$ be a non-empty set. Then there exist a unital vector lattice algebra $\FSETVLAONE{S}$ and a map $j:S \to\FSETVLAONE{S}$ with the property that, for any map $\varphi:S\to A$ from $S$ into a unital vector lattice algebra $\vlaone$, there exists a unique unital vector lattice algebra homomorphism $\fact{\varphi}$ such that the diagram
	\begin{equation*}
	\begin{tikzcd}
	S\arrow[r, "j"]\arrow[dr, "\varphi"]& \FSETVLAONE{S}\arrow[d, "\fact{\varphi}"]
	\\ & \vlaone
	\end{tikzcd}
	\end{equation*}
	is commutative. Furthermore:
	
	\begin{enumerate}
		\item the pair $(j,\FSETVLAONE{S})$ is unique up to a unique compatible isomorphism;
		\item $\FSETVLAONE{S}$ equals its unital vector lattice subalgebra that is generated $j(S)$;
		\item the map $j$ is injective.
	\end{enumerate}
\end{theorem}

For precisely the same reason\textemdash the existence of a category isomorphism with an equational class\textemdash it is clear that free vector spaces over non-empty sets exist (which can be seen much easier, of course), as do free vector lattices, free vector lattice algebras, and free vector lattice algebras with a positive identity element.

Likewise, the combination of \cref{res:two_types_of_lattices_are_equivalent} and \cref{res:free_algebra_with_relations_exists} shows that free lattices over non-empty sets exist. Adding the distributive laws to the identities in \cref{res:two_types_of_lattices_are_equivalent} shows that free distributive lattices over non-empty sets also exist.

Let us return to our original chain of categories
\begin{equation}\label{eq:chain_of_categories_again}
\SET \supset \VS \supset \VL \supset \VLA \supset \VLAONE \supset \VLAONEPOS.
\end{equation}
There are 15 instances of a category and a subcategory associated with this chain. For 5 of these, the ones  where $\SET$ has a subcategory, we know that there are always free objects of the subcategory because of the general theorem for equational classes.
How about the remaining 10?

For example, given a vector space $V$, do there exist a vector lattice $\FVSVL{V}$ and a linear map $j: V\to \FVSVL{V}$ with the property that, for every linear map $\varphi: E\to F$ from $E$ into a vector lattice $F$, there exists a unique vector lattice homomorphism $\fact{\varphi}:\FVSVL{V}\to F$ such that the diagram
\begin{equation*}
\begin{tikzcd}
V\arrow[r, "j"]\arrow[dr, "\varphi"]& \FVSVLA{V}\arrow[d, "\fact{\varphi}"]
\\	& F
\end{tikzcd}
\end{equation*}
is commutative?  As another example, given a vector lattice $E$, do there exist a unital vector lattice algebra $\FVLVLAONEPOS{E}$ with a positive identity element and a vector lattice homomorphism $j: E\to \FVSVLAONEPOS{E}$ with the property that, for every vector lattice homomorphism $\varphi: E\to \vlaonepos$ from $E$ into a vector lattice algebra $\vlaonepos$ with a positive identity element, there exists a unique unital vector lattice algebra homomorphism $\fact{\varphi}$ such that the diagram
\begin{equation*}
\begin{tikzcd}
E\arrow[r, "j"]\arrow[dr, "\varphi"]& \FVSVLAONEPOS{V}\arrow[d, "\fact{\varphi}"]\\ & \vlaonepos
\end{tikzcd}
\end{equation*}
is commutative?

As we shall see, the existence of all of these `missing' 10 free objects can be derived from the existence of $\FSETVLAONE{S}$ for non-empty sets $S$. We shall, in fact, also use this basic existence result to derive the existence of $\FSETVS{S}$, $\FSETVL{S}$, $\FSETVLA{S}$, and $\FSETVLAONEPOS{S}$ once more, even though we had already observed this to be a consequence of the general result for equational classes. The methods that are used below to obtain 14 other existence results from a basic one for a free object with `maximal' structure, can undoubtedly be formulated in general in terms of abstract algebras, their reducts (see \cite[p.7]{bergman_UNIVERSAL_ALGEBRA:2012}) and forgetful functors, inclusion of congruences relations and the general Second Isomorphism Theorem (see \cite[Theorem~3.5]{bergman_UNIVERSAL_ALGEBRA:2012}). We believe, however, that this would actually obscure the picture for the concrete cases we have in mind. Our approach, which is a very simple combination of passing to quotients and sub-objects, also leads to an overview of the (quite natural) relations between the various free objects of our interest; see \cref{res:overview_algebraic}, below. This overview would presumably be a little less obvious when using a more general abstract approach.

In order to solve the 14 remaining universal problems, we first make sure that free unital vector lattice algebras and free unital vector lattice algebras with positive identity elements over non-empty sets, over vector spaces, over vector lattices, and over vector lattice algebras all exist.  These are 8 universal problems in all.

The first batch of 4 free objects consists of the free unital vector lattice algebras in this list. We already have $\FSETVLAONE{S}$ for a non-empty set. This will be our starting point to construct free unital vector lattice algebras over vector spaces. To this end, let $V$ be a vector space. We let $\mathrm{Set\,} V$ be the underlying set of $V$ and take a pair $(j,\FSETVLAONE{\mathrm{Set\,} V})$. Suppose that $\varphi: V\to\vlaone$ is a linear map from $V$ into a unital vector lattice algebra $\vlaone$.  There exists a unique unital vector lattice algebra homomorphism $\fact{\varphi}:\FSETVLAONE{\mathrm{Set\,}V}\to \vlaone$ such that $\fact{\varphi}\circ j=\varphi$. For $x,y\in V$, we have, since $\varphi$ is actually linear, that
\begin{align*}
\fact{\varphi}(j(x+y)-j(x)-j(y))&=\fact{\varphi}(j(x+y))-\fact{\varphi}(j(x))-\fact{\varphi}(j(y))\\
&=\varphi(x+y)-\varphi(x)-\varphi(y)\\&=0.
\end{align*}
Likewise, one sees that $\fact{\varphi}(j(\lambda x)-\lambda j(x))=0$ for all $\lambda\in\RR$ and $x\in V$.
Hence $\fact{\varphi}$ vanishes on the bi-ideal $I$ of $\FSETVLAONE{\mathrm{Set\,}V}$ that is generated by the elements $j(x+y)-j(x)-j(y)$ for $x,y\in V$ and the elements  $j(\lambda x)-\lambda j(x)$ for $\lambda\in\RR$ and $x\in V$. Consider the quotient $\FSETVLAONE{\mathrm{Set\,} V}/I$, which is a unital vector lattice algebra, and let $q_I:\FSETVLAONE{\mathrm{Set\,}V}\to\FSETVLAONE{\mathrm{Set\,}V}/I$ be the quotient map. There exists a unique unital vector lattice homomorphism $\factfact{\varphi}: \FSETVLAONE{\mathrm{Set\,}V}/I\to \vlaone$ such that $\fact{\varphi}=\factfact{\varphi}\circ q_I$. Hence we have a commutative diagram
\begin{equation*}
\begin{tikzcd}
V\arrow[r, "j"]\arrow[dr,"\varphi"]& \FSETVLAONE{\mathrm{Set\,}V}\arrow[r,"q_I"]\arrow[d, "\fact{\varphi}"]&\FSETVLAONE{\mathrm{Set\,}V}/I\arrow[dl,"\factfact{\varphi}"]\\
&A&
\end{tikzcd}
\end{equation*}
Then $\factfact{\varphi}\circ (q_I\circ j)=\varphi$. Since $j(S)$ generates $\FSETVLAONE{\mathrm{Set\,}V}$ as a unital vector lattice algebra, $(q_I\circ j)(S)$ generates $\FSETVLAONE{\mathrm{Set\,V}}/I$ as a unital vector lattice algebra. This shows that the unital vector lattice homomorphism $\factfact{\varphi}$ is uniquely determined by the requirement that  $\factfact{\varphi}\circ(q_I\circ j)=\varphi$. Since, furthermore, $q_I\circ j: V\to \FSETVLAONE{\mathrm{Set\,}V}/I$ is linear, we see that the pair $(q_I\circ j, \FSETVLAONE{\mathrm{Set\,}V}/I)$ solves the problem of finding a free unital vector lattice algebra $\FVSVLAONE{V}$ over the vector space $V$.

The same method yields a free unital vector lattice algebra $\FVLVLAONE{E}$ over a vector lattice $E$. One starts with a pair $(j,\FSETVLAONE{\mathrm{Set\,}E}$, and lets $I$ be the bi-ideal in $\FSETVLAONE{\mathrm{Set\,}E}$ that is generated by the elements $j(x+y)-j(x)-j(y)$ for all $x,y\in E$, the elements $j(\lambda x)-\lambda j(x)$ for all $\lambda\in\RR$ and $x\in E$, and now also the elements $j(x\vee y)-j(x)\vee j(y)$ for all $x,y\in E$. The pair $(q_I\circ j, \FSETVLAONE{\mathrm{Set\,}E}/I)$ then solves the problem of finding a free unital vector lattice algebra $\FVLVLAONE{E}$ over the vector lattice $E$.

In order to obtain a free unital vector lattice algebra over a vector lattice algebra $\vla$, one includes the previous three classes of elements into the generating set of the ideal $I$ of $\FSETVLAONE{\mathrm{Set\,}A}$, and now also adds the elements $j(xy)-j(x)j(y)$ for all $x,y\in A$. The pair $(q_I\circ j, \FSETVLAONE{\mathrm{Set\,}A}/I))$ then  solves the problem of finding a free unital vector lattice algebra  $\FVLAVLAONE{\vla}$ over the vector lattice algebra $\vla$.

The second batch of 4 free objects from the list above consists of the free unital vector lattice algebras with positive identity elements over sets, over vector spaces, over vector lattices, and over unital vector lattice algebras.

Let $S$ be a non-empty set. Take a pair $(j,\FSETVLAONE{S})$. Let $I$ be the bi-ideal in $\FSETVLAONE{S}$ that is generated by $(\abs{1}-1)$, and let $q_     I:\FSETVLAONE{S}\to\FSETVLAONE{S}/I$ be the quotient map. Suppose that $\varphi: S\to \vlaonepos$ is a map from $S$ into a unital vector lattice algebra with positive identity element $\vlaonepos$. There exists a unique unital vector lattice algebra homomorphism $\fact{\varphi}: \FSETVLAONE{S}$ such that $\fact{\varphi}\circ j=\varphi$. Since the identity element in $\vlaonepos$ is positive, $\fact{\varphi}$ vanishes on $I$. Hence there exists a unique unital vector lattice algebra homomorphism $\factfact{\varphi}:\FSETVLAONE{S}/I\to \vlaonepos$ such that $\fact{\varphi}=\factfact{\varphi}\circ q_I$. Then $\FSETVLAONE{S}/I$ is a unital vector lattice algebra with positive identity element, and the pair $(q_I\circ j,\FSETVLAONE{S}/I)$ solves the problem of finding a free unital vector lattice algebra with positive identity element $\FSETVLAONEPOS{S}$ over the non-empty set $S$.

Analogously to this, one can, for a vector space $V$,  obtain $\FVSVLAONEPOS{V}$ as a quotient of $\FVSVLAONE{V}$ that we had already obtained. For a vector lattice $E$, $\FVLVLAONEPOS{E}$ is a quotient of $\FVSVLAONE{E}$ and, for a vector lattice algebra $\vla$, $\FVLAVLAONEPOS{\vla}$ is a quotient of $\FVLAVLAONE{\vla}$.

Alternatively, one can proceed as earlier, but now with the free unital vector lattice algebra with a positive identity element over a set as a starting point, instead of a free unital vector lattice algebra. For a vector space $V$, $\FVSVLAONEPOS{V}$ is then obtained as a quotient of $\FSETVLAONEPOS{\mathrm{Set\,}V}$;  for a vector lattice $E$, $\FVLVLAONEPOS{E}$ is then obtained as a quotient of $\FSETVLAONEPOS{\mathrm{Set\,}E}$; and, for a vector lattice algebra $\vla$, $\FVLAVLAONEPOS{\vla}$ is then obtained as a quotient of $\FSETVLAONEPOS{\mathrm{Set\,}A}$.

We have now completed our first task of obtaining all 8 universal objects in the above list. What about the remaining 7? As it turns out, it is easy to locate these. For this, we use \cref{res:embeddings}, which is concerned with the `reversal of directions' in our chain of categories in \cref{eq:chain_of_categories_again}. It enables us to locate the remaining 7 free objects.

Let $S$ be a non-empty set. Take a pair $(j,\FSETVLAONE{S})$. Suppose that $\vla$ is a vector lattice algebra, and that $\varphi:S\to \vla$ is a map. \cref{res:embeddings} shows that we may view $\vla$ as a vector lattice subalgebra of a unital vector lattice algebra $\vlaone$. After doing this, there exists a unique unital vector lattice algebra homomorphism $\fact{\varphi}:\FSETVLAONE{S}\to\vlaone$ such that $\varphi=\fact{\varphi}\circ j$.  Since $\fact{\varphi}$ maps $j(S)$ into the vector lattice subalgebra $\vla$ of $\vlaone$, this is also the case for the vector lattice subalgebra of $\FSETVLAONE{S}$ that is generated by $j(S)$. Hence this vector lattice subalgebra solves the problem of finding a free vector lattice algebra $\FSETVLA{S}$ over the non-empty set $S$.

Let $S$ be a non-empty set. Since \cref{res:embeddings} implies that every vector lattice can be viewed as a vector sublattice of a unital vector lattice algebra, a similar argument shows that $\FSETVL{S}$ is the vector sublattice of $\FSETVLAONE{S}$ that is generated by $j(S)$.

Let $S$ be a non-empty set. Since \cref{res:embeddings} implies that every vector space can be viewed as a vector subspace of a unital vector lattice algebra, it is now clear that $\FSETVS{S}$ is the vector subspace of $\FSETVLAONE{S}$ that is generated by $j(S)$.

We have thus taken care of another 3 free objects. Before proceeding, we note that \cref{res:embeddings} shows that non-empty sets, vector spaces, vector lattices, and unital vector lattice algebras can all be found inside some unital vector lattice algebra with a positive identity element. A similar argument, therefore, shows that, for a non-empty set $S$, each of $\FSETVLA{S}$, $\FSETVL{S}$, and $\FSETVS{S}$ can also be found inside $\FSETVLAONEPOS{S}$.

Let $V$ be a vector space. We have already shown that $\FVSVLAONE{V}$ exists. Using \cref{res:embeddings} again, it is then easily seen that $\FVSVLA{V}$ exists, and that it is the vector lattice subalgebra of $\FVSVLAONE{V}$ that is generated by $j(V)$. Likewise, $\FVSVL{V}$ is the vector sublattice of $\FVSVLAONE{V}$ that is generated by $j(V)$. We have thus found another 2 free objects. Before proceeding, we note that, for a vector space $V$,  $\FVSVLA{V}$ and $\FVSVLAONE{V}$ can both also be found inside $\FVSVLAONEPOS{V}$ again.

Let $E$ be a vector lattice. Using \cref{res:embeddings} again, we see that $\FVLVLA{E}$ exists and is equal to the vector lattice subalgebra of $\FVLVLAONE{E}$ that is generated by $j(E)$. We have now covered 14 free objects. Before proceeding, we note that $\FVLVLA{E}$ is also equal to the vector lattice subalgebra of $\FVLVLAONEPOS{E}$ that is generated by $j(E)$.

Finally, let $\vlaone$ be a unital vector lattice algebra. Let $I$ be the bi-ideal in $\vlaone$ that is generated by $(\abs{1}-1)$. It is then obvious that $\FVLAONEVLAONEPOS{\vlaone}$ exists and is simply $\vlaone/I$. We can also write this as $\FVLAONEVLAONE{\vlaone}/I$, thus making clear that this is completely analogous to the way how, for example, $\FVLVLAONEPOS{E}$ can be obtained as a quotient of $\FVLVLAONE{E}$ for a vector lattice $E$.

In the above, we have not mentioned the accompanying maps $j$ being injective or not. It is easy to see, using \cref{rem:embedding} and \cref{res:embeddings}, that these 15 maps are all injective.

We collect what we have found in the following theorem.

\begin{theorem}\label{res:overview_algebraic}
	For a non-empty set $S$, a vector space $V$, a vector lattice $E$, a vector lattice algebra $\vla$, and a unital vector lattice algebra $\vlaone$, the 15 free objects below all exist. There are inclusions as indicated. The surjective unital vector lattice algebra homomorphisms as indicated are the quotient maps corresponding to dividing out the bi-ideal that is generated by $(\abs{1}-1)$.
	\begin{equation*}
	\begin{tikzcd}[column sep=0.45em, row sep=-5pt]
	&&                      &                                             & \FSETVLAONE{S}\ar[dd, two heads]\\
	\,\,S\arrow[Subseteq]{r}{}&\FSETVS{S}\arrow[Subseteq]{r}{}&\FSETVL{S}\arrow[Subseteq]{r}{}&\FSETVLA{S}\arrow[Subseteq]{dr}{}\arrow[Subseteq]{ur}{}& \\
	&&                      &                                             &\FSETVLAONEPOS{S}\\
	\end{tikzcd}
	\end{equation*}
	\begin{equation*}
	\begin{tikzcd}[column sep=0.45em, row sep=-5pt]
	&&                      &                                             & \FVSVLAONE{V}\ar[dd, two heads]\\
	\phantom{\FSETVS{S}}&V\arrow[Subseteq]{r}{}&\FVSVL{V}\arrow[Subseteq]{r}{}&\FVSVLA{V}\arrow[Subseteq]{dr}{}\arrow[Subseteq]{ur}{}& \\
	&&                      &                                             &\FVSVLAONEPOS{V}\\
	\end{tikzcd}
	\end{equation*}
	\begin{equation*}
	\begin{tikzcd}[column sep=0.45em, row sep=-5pt]
	&&                      &                                             & \FVLVLAONE{E}\ar[dd, two heads]\\
	\phantom{\FSETVS{S}}&\phantom{\FVSVL{V}\,\,}&E\arrow[Subseteq]{r}{} &\FVLVLA{E}\arrow[Subseteq]{dr}{}\arrow[Subseteq]{ur}{}& \\
	&&                      &                                             &\FVLVLAONEPOS{E}\\
	\end{tikzcd}
	\end{equation*}
	\begin{equation*}
	\begin{tikzcd}[column sep=0.45em, row sep=-5pt]
	&&                      &                                             & \FVLAVLAONE{\vla}\ar[dd, two heads]\\
	\phantom{\FSETVS{S}}&\phantom{\FVSVL{V}}&\phantom{\FVSVL{V}\quad} &A\arrow[Subseteq]{dr}{}\arrow[Subseteq]{ur}{}& \\
	&&                      &                                             &\FVLAVLAONEPOS{\vla}\\
	\end{tikzcd}
	\end{equation*}
	\begin{equation*}
	\begin{tikzcd}[column sep=0.45em, row sep=-5pt]
	&&                      &                                             & \vlaone=\FVLAONEVLAONE{\vlaone}\ar[dd, two heads]\\
	\phantom{\FSETVS{S}}&\phantom{\FVSVL{V}}&\phantom{\FVSVL{V}} &\phantom{A}& \\
	&&                      &                                             &\FVLAONEVLAONEPOS{\vlaone}\\
	\end{tikzcd}
	\end{equation*}	
\end{theorem}

\begin{remark}
	For a vector lattice algebra $\vla$, $\FVLAVLAONE{\vla}$ is what deserves to be called the unitisation of $\vla$.
\end{remark}

\begin{remark}
	One can also ask for free \emph{commutative} vector lattice algebras, for free \emph{commutative} unital vector lattice algebras, and for free \emph{commutative} unital vector lattice algebras with a positive identity element. These can be obtained by taking the general free object and dividing out the bi-ideal that is generated by the elements $(j(x)j(y)-j(y)j(x))$ for all $x,y$ in the starting object. Using \cref{res:embeddings}, it is immediate that sets, vector spaces, and vector lattices still embed into these new free objects under the new map $j$, which is the composition of the quotient map and the original map $j$. Vector lattice algebras, however, embed precisely when they are commutative.
\end{remark}

\begin{remark}
	\cref{rem:composition} shows how compositions behave. For example, a free vector lattice algebra over a free vector lattice over a non-empty set $S$ is a free vector lattice algebra over $S$.
\end{remark}

\begin{remark}\label{rem:archimedean_free_object}
	There are 14 free objects in \cref{res:overview_algebraic} that are vector lattices or vector lattice algebras. These correspond to the 14 occurrences of a category and a subcategory in the chain in \cref{eq:chain_of_categories_again} where the subcategory consists of lattices. For each of these 14 occurrences, one can define a new subcategory by considering only Archimedean objects of the original subcategory. One can then ask for a free object in that context, i.e., ask for an Archimedean object of the original subcategory that has the universal property for all morphisms from the initial object of the category into Archimedean objects of the original subcategory. Using \cite[Theorem~60.2]{luxemburg_zaanen_RIESZ_SPACES_VOLUME_I:1971}, it is easy to see that the Archimedean free object is obtained from the general one in \cref{res:overview_algebraic} by taking its quotient with respect to the uniform closure of $\{0\}$.
	
	The general free (lattice) object in \cref{res:overview_algebraic} is sometimes already Archi\-me\-dean because it can be realised as a lattice of real-valued functions. This is the case for the free vector lattice $\FSETVL{S}$ over a non-empty set $S$ (see \cite{bleier:1973}) and for the free vector lattice $\FVSVL{V}$ over a vector space $V$ (this can be inferred from \cite[Theorem~3.1]{troitsky:2019}). We do not know whether any other of the remaining 12 is already Archimedean or not.
\end{remark}

\section{Free objects over lattices}\label{sec:free_objects_over_lattices}

\noindent The free objects obtained above, and the two methods of obtaining new ones that were used above (passing to quotients and sub-objects), can also be used in other context. As an example, consider the chain of categories
\[
\SET\supset\LAT\supset\VL\supset\VLA\supset\VLAONE\supset\VLAONEPOS.
\]
We have already observed that the free lattice over a non-empty subset exists, as a consequence of \cref{res:free_algebra_with_relations_exists}. How about the other four existence problems that are not related to the original chain in \cref{eq:chain_of_categories}? We shall now show that the corresponding free objects all exist, as an easy consequence of our previous work. Suppose that $L$ is a (not necessarily distributive) algebraic lattice. Take $\FSETVLAONE{\mathrm{Set}\,L}$, and let $I$ be the bi-ideal in $\FSETVLAONE{\mathrm{Set}\,L}$ that is generated by the elements $(j(x\wedge y)-j(x)\wedge j(y))$ and $(j(x\vee y)-j(x)\vee j(y))$ for all $x,y\in L$. Then $\FSETVLAONE{\mathrm{Set}\,L}/I$  is the free unital vector lattice algebra $\FLATVLAONE{L}$ over $L$. The quotient of  $\FLATVLAONE{L}$ modulo the bi-ideal that is generated by $(\abs{1}-1)$ is the free unital vector lattice algebra with a positive identity element $\FLATVLAONEPOS{L}$ over $L$.	Using \cref{res:embeddings}, one sees that $\FLATVL{L}$ and $\FLATVLA{L}$ are the vector sublattice, resp.\ the vector lattice subalgebra, of $\FLATVLAONE{S}$ (and also of $\FLATVLAONEPOS{S}$) that is generated by $j(L)$.

It is not always the case that the maps $j:L\to\FLATVL{L}$, $j:L\to\FLATVLA{L}$, $j: L\to\FLATVLAONE{L}$, or $j:L\to\FLATVLAONEPOS{L}$ are injective. If one of these is injective, then $L$ must be distributive. Conversely, suppose that $L$ is distributive. Then $L$ is isomorphic to a sublattice of the power set of some set $X$; see \cite[Corollary~2.42]{bergman_UNIVERSAL_ALGEBRA:2012} or \cite[Theorem~119]{gratzer_LATTICE_THEORY_FOUNDATION:2011}, for example. Passing to characteristic functions, we see that a distributive lattice $L$ is isomorphic to a sublattice of the lattice of real-valued bounded functions on $X$.  As earlier, since the real-valued bounded functions on $X$ form a unital vector lattice algebra with a positive identity element, the existence of this single (non-trivial) embedding of a distributive lattice $L$ already implies that all four maps $j$ are injective.
Thus we have shown the following.

\begin{theorem}\label{res:free_objects_over_lattices}
	Let $L$ be a \uppars{partially ordered or algebraic} lattice. Then the 4 free objects in
	
	\begin{equation*}
	\begin{tikzcd}[column sep=1.5em, row sep=-5pt]
	&&                      &                                             & \FLATVLAONE{L}\ar[dd, two heads]\\
	&L\ar[r, "j"]&\FLATVL{L}\arrow[Subseteq]{r}{}&\FLATVLA{V}\arrow[Subseteq]{dr}{}\arrow[Subseteq]{ur}{}& \\
	&&                      &                                             &\FLATVLAONEPOS{L}\\
	\end{tikzcd}
	\end{equation*}
	all exist. The map $j$ is injective if and only if $L$ is distributive.
\end{theorem}

\begin{remark}
	The above argument linking the injectivity of $j$ to the distributivity of $L$ is taken from \cite[proof of Proposition~3.1]{aviles_rodriguez-abellan:2019}, where it was used in the context of free Banach lattices over a lattice.  The bounded real-valued functions on $X$, when supplied with the supremum norm, also form a unital Banach lattice algebra with a positive identity element. The fact that the distributive lattice $L$ embeds into its unit ball will have consequences for the injectivity of the maps $j$ when considering free Banach lattice algebras over distributive lattices.
\end{remark}

\begin{remark}
	As in \cref{rem:archimedean_free_object}, a free Archimedean object can be obtained by taking the quotient of a general free object in \cref{res:free_objects_over_lattices} with respect to the uniform closure of $\{0\}$. It can be inferred from \cite[Theorem~2.1]{aviles_rodriguez-abellan:2019} that the free vector lattice $\FLATVL{L}$ over a (not necessarily distributive) lattice $L$ can be realised as a vector lattice of real-valued functions. Hence it is Archimedean. We do not know whether any of the other three free objects in \cref{res:free_objects_over_lattices} is already Archimedean or not.
\end{remark}

\section{Free $\!f\!$-algebras}\label{sec:free_f-algebras}

\noindent We conclude with a discussion of free $\!f\!$-algebras over non-empty sets. We can neither prove nor disprove that they exist, but there are still a number of observations to be made.

We recall that a vector lattice algebra $\vla$ is a member of a family of abstract algebras, each of which is supplied with a constant 0, a binary map $\oplus$, a unary map $\ominus$, a unary map $m_\lambda$ for every $\lambda\in\RR$, a binary map $\odot$, and binary maps $\alwedge$ and $\alvee$. We let $\type$ denote the obvious underlying type of such abstract algebras. Among all abstract algebra of this type $(\mathcal F, \rho)$, we can single out the vector lattice algebras as those in which a number of identities are satisfied: they form an equational class. We have seen that this implies that free vector lattice algebras over a non-empty set exist.

How is this with $\!f\!$-algebras?
We recall that a vector lattice algebra is called an $\!f\!$-algebra if, for all $x,y\in A$ and $z\in A^+$, the fact that $x\wedge y=0$ implies that $(xz)\wedge y=(zx)\wedge y=0$. We can, therefore, also single out the $\!f\!$-algebras among all abstract algebras of type $\type$ by requiring that all the identities for vector lattice algebras be again satisfied, and requiring that this extra $\!f\!$-algebra implication be valid. Is this perhaps also an equational class? This would imply the existence of free $\!f\!$-algebras over non-empty sets. Likewise, if one can prove that the  Archimedean $\!f\!$-algebras form an equational class, then this would establish the existence of free Archimedean $\!f\!$-algebras.

For the unital case, a similar setup can be given. One starts with abstract algebras, each of which is supplied with constants 0 and now also 1,  a binary map $\oplus$, a unary map $\ominus$, a unary map $m_\lambda$ for every $\lambda\in\RR$, a binary map $\odot$, and binary maps $\alwedge$ and $\alvee$. There is an obvious underlying type again (slightly different from the previous one), and the unital vector lattice algebras are the abstract algebras of this type in which a (slightly different) number of identities are satisfied. They form an equational class and, therefore, free unital vector lattice algebras over non-empty sets exist.

How is this with unital $\!f\!$-algebras, which can be singled out as those unital vector lattice algebras where the $\!f\!$-algebra implication holds? Do they form an equational class? If so, then free unital $\!f\!$-algebras over non-empty sets exist. How about Archimedean unital $\!f\!$-algebras, unital $\!f\!$-algebras with a positive identity element, and Archimedean unital $\!f\!$-algebras with a positive element?

All in all, we have six classes of $\!f\!$-algebras that may or may not be equational classes. For three of them, we can show that they are not. The reason is that equational classes are closed under the taking of abstract homomorphic images, i.e., under the taking of vector lattice algebra homomorphic images. For the Archimedean $\!f\!$-algebras, the Archimedean unital $\!f\!$-algebras, and the Archimedean unital $\!f\!$-algebras with a positive identity, this is not the case as is demonstrated by the following example. It is based on \cite[second part of Example~60.1]{luxemburg_zaanen_RIESZ_SPACES_VOLUME_I:1971}, where it is used to show that a quotient of an Archimedean vector lattice need not be Archimedean. The particular context shows much more, however.

\begin{example}
	Under pointwise algebra operations and ordering, the sequence space $\ell^\infty$ is a unital $\!f\!$-algebra with a positive identity element. Consider the order ideal $I_u$ of $\ell^\infty$ that is generated by the element $u=(u_1,u_2,u_3\dotsc)\coloneqq (1/1^2, 1/2^2, 1/3^2,\ldots)$. Since the elements of $\ell^\infty$ are bounded,  $I_u$ is also an algebra ideal. Hence it is a bi-ideal, so that the quotient $\ell^\infty/I_u$ is a vector lattice algebra again. This quotient is not Archimedean, however. To see this, we include the short argument from  \cite[Example~60.1]{luxemburg_zaanen_RIESZ_SPACES_VOLUME_I:1971}. Let $q:\ell^\infty\to\ell^\infty/I_u$ denote the quotient map. Set $e\coloneqq (1,1,1,\dotsc)$ and $v=(v_1,v_2,v_3,\dotsc)\coloneqq (1/1,1/2,1/3,\dotsc)$. Then $q(v)\neq 0$.  Let $k$ be a positive integer. Then $v_n\leq e_n/k$ for all $n\geq k$. Hence there exists an element $x=(x_1,x_2,\dotsc)\in\ell^\infty$ such that $x_n=0$ for all $n\geq k$ and $v\leq x + e/k$. Since $x\in I_u$, this implies that $q(v)\leq q(e)/k$. Since $q(v)>0$, this shows that $\ell^\infty/I_u$ is not Archimedean.
	
	We therefore have a vector lattice algebra homomorphism (which is automatically unital by its surjectivity)  $q:\ell^\infty\to\ell^\infty/I_u$, but the codomain is not even an Archimedean vector lattice, let alone an Archimedean $\!f\!$-algebra with or without additional properties.
	
	Aside, although it is not relevant to our main issues, let us nevertheless note that the vector lattice algebra $\ell^\infty/I_u$ is not just a vector lattice algebra, but even an $\!f\!$-algebra. To see this, suppose that $x,y\in\ell^\infty$ are such that $q(x)\wedge q(y)=0$. We may suppose that $x,y\geq 0$. Take a positive element $q(z)$ of $\ell^\infty/I_u$, where  $z=(z_1,z_2,\dotsc)\in\ell^\infty$. We may suppose that $z\geq 0$. Since $q(x\wedge y)=q(x)\wedge q(y)=0$, we have $x\wedge y\in I_u$. Then the estimate
	\[
	0\leq ((xz)\wedge y)_n=(x_nz_n)\wedge y_n\leq (\norm{z}_\infty+1) (x_n\wedge y_n)
	\]
	for all $n$ shows that $(xz)\wedge y\in I_u$. Hence $(q(x)q(z))\wedge q(y)=q((xz)\wedge y)=0$, as required.
	
	It seems worthwhile to note explicitly that losing the Archimedean property when passing to a homomorphic image is not only possible in the category of vector lattices, but also in the much smaller subcategory of unital $\!f\!$-algebras with a positive identity element.
	
	\begin{proposition}
		There exist an Archimedean unital $\!f\!$-algebra $\vlaonepos$ with a positive identity element, a non-Archimedean $\!f\!$-algebra $B$ \uppars{automatically unital with a positive identity element} and a surjective vector lattice algebra homomorphism $\varphi:A\to B$.
	\end{proposition}
	
\end{example}

Returning to the main line, let us remark that, for the remaining three classes of $\!f\!$-algebras, we do not know whether they are equational classes or not. We do not have an example showing that they are not, but our attempts to `equationalise' the validity of the $\!f\!$-algebra implication have failed. The latter does, of course, not show that it is impossible to do so.

A more indirect way to prove that they \emph{are} equational classes would be to try to use Birkhoff's theorem (see \cite[Theorem~4.41]{bergman_UNIVERSAL_ALGEBRA:2012}), which shows that the equational classes are precisely the varieties of abstract algebras. We recall that a variety of abstract algebras is a class of abstract algebras, all of the same type, that is closed under taking abstract subalgebras, abstract product algebras, and abstract homomorphic images. It is clear that each of the classes of $\!f\!$-algebras, of unital $\!f\!$-algebras, and of unital $\!f\!$-algebras with a positive identity element are closed under taking (unital) vector lattice subalgebras and taking vector lattice algebraic products. It seems to be open, however, whether any of these three classes is closed under the taking of abstract algebra homomorphic images, i.e., under (unital) vector lattice algebra homomorphic images. Hence we have the following questions.

\begin{questions}\label{ques:f_algebras}
	Let $\vla$ be an $\!f\!$-algebra, let $B$ be a vector lattice algebra, and let $\varphi: A\to B$ be a surjective vector lattice algebra homomorphism.
	\begin{enumerate}
		\item Is $B$ always an $\!f\!$-algebra? If so, then the $\!f\!$-algebras form an equational class and free $\!f\!$-algebras over non-empty sets exist.
		
		\item Is $B$ always an $\!f\!$-algebra \uppars{automatically unital} when $\vla$ is unital? If so, then the unital $\!f\!$-algebras form an equational class and free unital $\!f\!$-algebras over non-empty sets exist.

		\item Is $B$ always an $\!f\!$-algebra \uppars{automatically unital with a positive identity element} when $\vla$ is unital with a positive identity element? If so, then the unital $\!f\!$-algebras with a positive identity element form an equational class and free unital $\!f\!$-algebras with a positive identity element over non-empty sets exist.
	\end{enumerate}
\end{questions}

We have no answers. Related to all three questions, we can only mention that it is known that the answers are affirmative, \emph{provided} that we also know that both $\vla$ and $B$ are Archimedean; this follows from \cite[Proposition~3.2]{boulabiar:2002}. Related to the second and third question, we can only mention that it is known that an Archimedean vector lattice algebra with an identity element is an $\!f\!$-algebra precisely when all squares are positive;  see \cite[Corollary~1]{steinberg:1976}. This shows once more that the answers to the second and third questions are affirmative, \emph{provided} that we also know that both $\vla$ and $B$ are Archimedean. Unfortunately, this is not what we need, and to the best of our knowledge these three issues are all open at the time of writing.

Even though the Archimedean $\!f\!$-algebras, the Archimedean unital $\!f\!$-algebras, and the Archimedean unital $\!f\!$-algebras with a positive identity are not an equational class, and the $\!f\!$-algebras, the unital $\!f\!$-algebras, or the unital $\!f\!$-algebras with a positive identity element might also not be, this does not preclude the possibility that one or more of these six classes still has free objects over non-empty sets. After all, the Archimedean vector lattices do not form an equational class because there exist quotients of such lattices that are no longer Archimedean, but free Archimedean vector lattices over non-empty sets still exist. The reason is simply that the general free vector lattice over a non-empty set, which has a model as a lattice of real-valued functions,  just happens to be Archimedean. This fact does not appear to be accessible with methods from universal algebra or category theory alone; one really has to look at the internal structure of the free object once one knows that it exists as is done in \cite{bleier:1973}, for example. It is conceivable that something similar may be the case for $\!f\!$-algebras, where free vector lattice algebras of various kinds over non-empty sets\textemdash the existence of which is guaranteed by the general result for equational classes\textemdash may happen to be objects of a much smaller subcategory of $\!f\!$-algebras, where they are then evidently also free objects over non-empty sets. The following questions are, therefore, natural to ask:

\begin{questions}
	Let $S$ be a non-empty set.
	\begin{enumerate}
		\item Is $\FSETVLA{S}$ an $\!f\!$-algebra? Is it Archimedean?
		\item Is $\FSETVLAONE{S}$ an $\!f\!$-algebra? Is it Archimedean?
		\item Is $\FSETVLAONEPOS{S}$ an $\!f\!$-algebra? Is it Archimedean?
	\end{enumerate}
\end{questions}

Of course, if $\FSETVLAONE{S}$ or $\FSETVLAONEPOS{S}$ is Archimedean, or an $\!f\!$-algebra, then the same is true for its vector lattice subalgebra $\FSETVLA{S}$.


\subsection*{Acknowledgements} It is a pleasure to thank Karim Boulabiar, Ben de Pagter, Mitchell Taylor, and Vladimir Troitsky for helpful discussions.





\bibliography{general_bibliography}


\end{document}